\documentclass[10pt]{amsart} 

\usepackage[psamsfonts]{amssymb} 
\usepackage{amsfonts,amsmath} 
\usepackage{graphicx, color}
\usepackage{enumerate}
\usepackage{url}
\usepackage{tikz-cd}
\usepackage[hidelinks]{hyperref}
\usepackage{mathtools}
\usepackage{stmaryrd}

\usepackage{amsthm, amssymb} 

\usepackage{gradient-text}

\usepackage{stix}

\usepackage[all]{xy}

\newtheorem{thmA}{Theorem}

\newtheorem{corA}[thmA]{Corollary}

\newtheorem{theorem}{Theorem}[section]
\newtheorem{prop}[theorem]{Proposition}
\newtheorem{proposition}[theorem]{Proposition}
\newtheorem{lemma}[theorem]{Lemma}
\newtheorem{corollary}[theorem] {Corollary}

\newtheorem*{claim*}{Claim}

\newtheorem{warning}[theorem]{Warning}

\theoremstyle{remark}
\newtheorem{remark}[theorem]{Remark}

\theoremstyle{definition}
\newtheorem{definition}[theorem]{Definition}
\newtheorem{notation}[theorem]{Notation}

\newcommand{\tn}[1]{\textnormal{#1}}
\newcommand{\cd}{\textnormal{cd}}

\newcommand{\ring}{\Lambda}
\newcommand{\rg}{\ring(G)}

\newcommand{\Hom}{\tn{Hom}}

\newcommand{\solid}{\smblksquare}

\DeclareMathOperator{\Solid}{Solid}
\newcommand{\LSolid}{\Solid_\ring}

\newcommand{\SolidLG}{\Solid_{\rg}}

\newcommand{\Z}{\mathbb{Z}}

\newcommand{\rar}{\rightarrow}

\newcommand{\rh}{\tn{R}\underline{\tn{Hom}}_\ring}
\newcommand{\rhg}{\tn{R}\underline{\tn{Hom}}_{\rg}}

\newcommand{\cG}{\underline{G}}
\newcommand{\cR}{{\underline{\ring}}}

\newcommand{\st}{\otimes_{\ring}^{\solid}}
\newcommand{\stheart}{\otimes_{\ring, \heartsuit}^{\solid}}
\newcommand{\stg}{\otimes_{\rg}^{\solid}}

\newcommand{\Ind}{\operatorname{Ind}}
\newcommand{\colim}{\operatorname{colim}}

\DeclareMathOperator{\Pro}{Pro}

\newcommand{\finitemodules}{\mathcal{F}_{\ring\llbracket G \rrbracket}}

\newcommand{\cgr}[1]{\ring\llbracket #1 \rrbracket}

\newcommand{\id}{\tn{id}}

\title{Solid Duality for Profinite Groups}
\author{Ged Corob Cook, Max Gheorghiu, Sof\'ia Marlasca Aparicio, Thomas A. Wasserman}   
\date{\today}
\begin{document}
	
\begin{abstract} 
We classify profinite groups that have a Poincar\'e-like duality between their homology and cohomology. Our proofs work over every profinite coefficient ring, and for profinite as well as discrete coefficients. We do not only unify and generalise existing results, but also construct two novel examples of duality groups.

We prove our results via the framework of condensed mathematics. This recently developed setting is a natural home for profinite objects, and our work is among the first to leverage this for the study of the (co)homology of profinite groups. Establishing our duality results involves an investigation of homological finiteness properties in condensed mathematics.
\end{abstract}
	
\address{}
\email{}

\maketitle

\setcounter{tocdepth}{1}
\tableofcontents

\section{Introduction}
Our main results concern a Poincaré-like duality between homology and cohomology of a profinite group $G$. We assume $G$ to have finite cohomological dimension $n$ over a profinite coefficient ring $\ring$ such that its cohomology $H_{\ring}^{\bullet}(G, M)$ is finite for every finite module $M$ over the completed group ring $\ring\llbracket G\rrbracket$ (condition (FF), standing for `finite-to-finite'). Our main tool is \emph{condensed mathematics}, a powerful unified approach for studying topological groups, rings and modules that has been introduced by Clausen and Scholze in~\cite{Scholze2019}. Within this framework, associated with $G$ and $\ring$ are a condensed group $\cG$, as well as the \emph{analytic rings} $\cR$ and $\rg$ with corresponding \emph{solid modules}. We say that $G$ is a \emph{solid duality group} over $\cR$ if there is a \emph{solid dualising module} $D_{\cR}$ and an isomorphism
 $$
\mathbf{H}^{\cR}_k(G, D_{\cR}\stg N) \cong \mathbf{H}_{\cR}^{n-k}(G, N),
$$
for the \emph{solid (co)homology} of $G$ for every degree $k$ and all solid $\rg$-modules $N$.

We prove that $G$ is a solid duality group if and only if $G$ is \emph{Cohen--Macaulay} over $\ring$. Here, the Cohen--Macaulay condition requires that the continuous cohomology of $G$ satisfies
$$
H^k(G, \ring\llbracket G\rrbracket)=\begin{cases}
    D_\ring&\text{for } k=n, \\ 0& \text{otherwise},
\end{cases}
$$
where $D_\ring$ is a $\ring$-flat $\ring\llbracket G\rrbracket$-module. Moving from the condensed to the profinite setting, we use this to conclude that a profinite group with condition (FF) has a duality isomorphism for profinite modules if and only if it is Cohen--Macaulay. Under stronger conditions, this duality isomorphism in the profinite setting can extended to all discrete modules.

Along the way we investigate homological finiteness conditions. We establish the first example of a profinite duality group that is not of type $\mathrm{FP}_\infty$. We show that finitely generated free pronilpotent groups have duality over $\widehat{\Z}$. We also prove that profinite duality groups close under taking open subgroups, extensions and certain amalgamated products.

Our results are among the very first in group theory that have only been achieved through condensed mathematics.

\subsection{Background}\label{ssec:background}
Serre, Verdier and Tate \cite{serre1964cohomologie,Verdier1965,tat94} introduced versions of Poincar\'e-like duality for a profinite group $G$ which were formulated over the profinite coefficient rings $\Z/p{\Z}$, $\Z_p$ and $\widehat{\Z}$. To state this duality, assume that $G$ has finite cohomological dimension $\cd_{\Z_p} (G) = n$ and satisfies condition (FF). If $(-)^{\ast}$ denotes the Pontryagin duality operator, then this duality over $\Z_p$ consists of a profinite \emph{dualising module} $D_{\Z_p}$ and of isomorphisms
\[ H_{\Z_p}^k(G, \mathrm{Hom}(M, D_{\Z_p}^{\ast})) \rightarrow H_{\Z_p}^{n-k}(G, M)^{\ast} \]
for every finite $\Z_p \llbracket G \rrbracket$-module $M$ in every degree $k$. Verdier~\cite{Verdier1965} proved that a profinite group has duality over $\Z_p$ in this sense if and only if it is Cohen--Macaulay.

A discrete group $\Gamma$ is said to have \emph{Bieri--Eckmann duality \cite{bie73}} over a ring $R$ if $\tn{cd}_R(\Gamma)=n$, there exists a dualising module $D_R$, and for every $R[\Gamma]$-module $M$ there are isomorphisms in every degree $k$ induced by a cap product
\[ H_k^R(\Gamma, M \otimes_{R} D_{R}) \rightarrow H_R^{n-k}(\Gamma, M). \]
Pletch showed that Serre--Verdier--Tate duality is equivalent to the profinite analogue of Bieri--Eckmann duality for finite $\Z_p \llbracket G \rrbracket$-modules \cite{Pletch1980}, and that Bieri--Eckmann duality passes to the profinite completion of cohomologically good groups \cite{Pletch1980a}. Inspired by work of Lazard \cite{lazard1965groupes}, Symonds and Weigel \cite{sym00} introduced the Pontryagin category containing both discrete and profinite $\Z_p \llbracket G \rrbracket$-modules. For $G$ satisfying a homological finiteness condition called type $\mathrm{FP}_{\infty}$ over $\Z_p$, they formulated homology and cohomology of $G$ taking coefficients in the Pontryagin category. Further assuming $\cd_{\Z_p} (G) = n$, they extended Bieri--Eckmann duality over $\Z_p$ from finite to profinite $\Z_p \llbracket G \rrbracket$-modules. In the particular case where the dualising module is also trivial, they showed duality extends to the entire Pontryagin category. This perspective underlies the current state of the art and has been refined by Wilkes~\cite{Wilkes2019} who continued to consider duality only over $\Z_p$. 

This classical framework bears some inconveniences. One of them is that duality has only been formulated over particular examples of profinite coefficient rings. Even if one continues to only consider duality over $\Z_p$, one has to take due care of discrete and profinite modules. If one places them in distinct categories, then discrete modules lack enough projectives while profinite modules lack enough injectives, resulting in an asymmetric footing. One can embed discrete and profinite coefficients in the Pontryagin category, but this category is not abelian and lacks both enough injectives and projectives.

Condensed mathematics treats both types of modules on the same footing: as objects in an abelian category of \emph{solid modules} that has enough injectives and projectives. The homology and cohomology of a profinite group with coefficients in solid modules were first introduced in unpublished work by Ansch\"utz and Le Bras~\cite{Anschuetz2020}, shortly after condensed mathematics and the theory of solid modules had been introduced in~\cite{Scholze2019}. The theory of (co)homology of profinite groups with solid coefficients was given further substance by Tang \cite{tang24} and Brink \cite{bri25}. Tang shows that whenever they are comparable, continuous (co)homology and solid (co)homology agree with each other. We extract from this the corresponding statement for the derived functors (Theorem~\ref{thm:profintosolid}). Ansch\"utz and Le Bras studied duality for profinite groups using solid modules, but only over particular examples of profinite coefficient rings and only treated the case where the dualising module is trivial. The duality we establish (Theorem~\ref{thmA} and Theorem~\ref{thmB}) has neither of these limitations. We generalise Bieri--Eckmann duality with its cap product and Serre--Verdier--Tate duality to all profinite coefficient rings (Proposition~\ref{prop:capduality} and Proposition~\ref{prop:cupduality}).

To underline the generality of our results, we construct two novel examples of duality groups that exceed previous limitations. Thus far, every existing example of a profinite duality group $G$ over $\ring$ satisfies a homological finiteness condition called type $\mathrm{FP}_{\infty}$. We construct the first example of a profinite duality group over $\mathbb{Z}_p$ that is not of type $\mathrm{FP}_{\infty}$ (Theorem~\ref{thm:novelexample}). As in~\cite{Pletch1980}, \cite{ser94} and \cite{Verdier1965}, our duality only needs the weaker condition (FF), which has been also termed Serre's condition (F): that $H_{\ring}^k(G, M)$ is finite for every finite $\ring \llbracket G \rrbracket$-module $M$ in any degree $k$. Type $\mathrm{FP}_{\infty}$ implies condition (FF) by~\cite[Proposition~4.2.2]{sym00}. Furthermore, we demonstrate that finitely generated free pronilpotent groups are duality groups over $\widehat{\mathbb{Z}}$ (Proposition~\ref{prop:freegroups}).

Pletch also investigated various group-theoretic properties closure properties of duality groups in~\cite{Pletch1980} and~\cite{Pletch1980a}. This includes taking open subgroups of a profinite group and proving that the subgroup satisfies duality if and only if the overgroup does so. The other properties are studied in a similar vein. As stated, Pletch proved these results only for duality groups over $\Z_p$. We generalise three of their results to arbitrary duality groups over any profinite coefficient ring: they are closed under taking open subgroups (Proposition~\ref{prop:subgroupclosed}), taking extensions (Proposition~\ref{lem:extensionclosed}) and forming profinite amalgamated products (Propostion~\ref{lem:amalgamclosed}). In particular, the last two properties allow one to construct new examples of duality groups by taking subnormal series or iterated amalgamated products.

We provide one of the first accounts of homological finiteness properties in condensed mathematics. To our knowledge, the only other treatment of finiteness properties is found in Ansch\"utz and Le Bras \cite[Section 3]{Anschuetz2020}, whose main results we generalise. More specifically, Ansch\"utz and Le Bras only work over particular examples of profinite coefficient rings, while our results pertain to every profinite coefficient ring $\ring$. Vastly generalising \cite[Proposition~3.9]{Anschuetz2020}, we prove that the cohomological dimension of a profinite group $G$ for profinite modules $\cd_{\ring} (G)$ equals that for solid modules $\cd_{\underline{\ring}}(G)$ (Theorem~\ref{thm:cd}). Recall that an object $X$ in a category is called compact if $\mathrm{Hom}(X, -)$ commutes with direct limits (i.e. filtered colimits). Generalising \cite[Lemma~3.1]{Anschuetz2020}, we prove that any profinite group $G$ is of type $\mathrm{CP}_{\infty}$ over $\underline{\ring}$ (Theorem~\ref{thm:cpoo}), meaning that there is a projective resolution $P_{\bullet} \rightarrow \underline{\ring}$ of solid $\rg$-modules such that every term $P_i$ is compact.

The core technical result (Lemma~\ref{lem:corelemma}) underlying our duality is a statement about the \emph{derived solid tensor product} $\stg$ and the \emph{derived solid internal Hom} $\rhg(-,-)$. For every profinite group $G$ of finite cohomological dimension satisfying condition (FF), it says that the tensor product with the \emph{dualising complex} $D^\bullet_{\cR}:= \rhg(\cR, \rg)$ is isomorphic to mapping out of $\cR$:
\begin{equation}\label{eq:key iso}
     D_{{\cR}}^{\bullet} \stg(-)  \xrightarrow{\cong} \rhg(\cR, -) \, .
\end{equation}
Together with an identity of tensor products,  this implies (Theorem~\ref{thm:tate}) an isomorphism
\begin{equation}\label{eq:tate iso}
     \cR\stg\left(D_{{\cR}}^{\bullet} \st(-) \right) \xrightarrow{\cong} \rhg(\cR, -) \, .
\end{equation}
The homology of $\cR\stg (-)$ is the \emph{solid group homology} $\mathbf{H}_{\bullet}^{\cR}(G, -)$ and that of $\rhg(\cR,-)$  is the \emph{solid internal group cohomology} $\mathbf{H}_{\cR}^{\bullet}(G, -)$.  Associated with the isomorphism \eqref{eq:tate iso} is a spectral sequence relating solid group homology and cohomology analogous to the classical Tate spectral sequence \cite[Theorem 2.5.3]{neu13} for discrete groups. As in the classical case, a profinite group is a duality group if and only if the solid version of the Tate spectral sequence collapses on the $E_2$ page.

\subsection{Framework}
\subsubsection{Condensed mathematics}
We provide a brief sketch of our main tool, Condensed mathematics. Its building blocks are condensed sets, which are sheaves of sets on profinite spaces (up to a set-theoretic caveat). There is an operation called \emph{condensation} that turns any $T1$-topological space $X$ (meaning that all its points are closed) into a condensed set $\underline{X}$ mapping any profinite space $S$ to $\Hom_{\tn{cts}}(S,X)$. If $X$ is additionally a topological group, ring, module, etc., the pointwise operations turn $\underline{X}$ into a condensed group, ring, module, etc. Condensation is fully faithful on the subcategory of profinite spaces by the Yoneda lemma. This naturally embeds profinite algebraic objects into the condensed world. 

Contrary to topological modules over topological rings, condensed modules possess excellent category-theoretic properties~\cite[Theorem~1.10]{Scholze2019}. Namely, they form an abelian category in which all limits and colimits exist. Arbitrary products, direct sums and direct limits are exact. These categories are generated by compact projective objects, meaning that every condensed module admits a surjective homomorphism from a coproduct of these compact projectives. In particular, condensed modules have enough projectives.

Condensates of profinite rings as well as condensates of discrete and profinite modules over them are particularly nice: they are canonically \emph{analytic} rings and  \emph{solid} modules over them, respectively. An analytic ring can be thought of as a condensed ring with a choice of ``complete'' free modules; solid modules satisfy a form of completeness with respect to this choice. Condensation of profinite and discrete modules is fully faithful and exact~\cite[Theorem~3.2]{tang24}, and furthermore condensates of projective profinite modules are projective solid modules~\cite[Theorem~3.14]{tang24}. Therefore, homological algebra with continuous modules can be equivalently performed using solid modules. This is the content of Theorem~\ref{thm:profintosolid}, which lifts Tang's comparison of continuous (co)homology and solid (co)homology \cite{tang24} to the derived funtors:
\begin{align}\begin{split}\label{eq:derived functors comparison}
     \underline{M \widehat{\otimes}_{\ring} N}& \cong \underline{M} \st \underline{N}\\
     \underline{M \widehat{\otimes}_{\ring \llbracket G \rrbracket} N} &\cong \underline{M} \stg \underline{N}\\
     \mathrm{RHom}_{\ring \llbracket G \rrbracket}(M, N)& \cong \mathrm{RHom}_{\rg}(\underline{M}, \underline{N})\\
    \underline{\mathrm{RHom}_{\ring \llbracket G \rrbracket}(M, N)}& \cong \rhg(\underline{M}, \underline{N}).\end{split}
\end{align}
This explains how we can uniformly treat profinite and discrete modules.

\subsubsection{Derived categories and sifted colimits}
The natural context for our work is that of derived categories, in particular the derived category $\SolidLG$ of solid modules over the analytic group ring $\rg$. One of the key advantages of condensed mathematics is that every object in $\SolidLG$ can be written as a sifted colimit of complexes of compact projectives \cite[Proposition 12.4]{scholze2026}. Sifted colimits can be thought of as a combination of direct limits and reflexive coequalisers~\cite[p.~251]{ada10}; direct limits are sifted. For $\SolidLG$ the compact projective generators can be taken to be objects of the form  $\prod_I \rg$ where $I$ ranges over all sets~\cite[Proposition~3.10]{tang24}.

Generation by compact projectives under sifted colimits suggests a strategy of proof for showing that a natural transformation between two functors is a natural isomorphism: first prove that both functors commute with sifted colimits, and then show the natural transformation is an isomorphism on compact projectives. In the case of $\SolidLG$, there is also an obvious strategy for the latter: show both functors commute with arbitrary products and then check the natural transformation is an isomorphism on the free rank-one module $\rg$. This final step is usually straightforward.

We want to apply this strategy to establish the core isomorphism \eqref{eq:key iso}. Some of this is formal: the tensor product functor $D^\bullet_{\cR} \stg -$ commutes with all small colimits as it is a left adjoint; the functor $\rhg(\cR,-)$ commutes with all small limits as it is a right adjoint, so in particular commutes with arbitrary products. Ensuring that $D^\bullet_{\cR} \stg -$ commutes with arbitrary products and $\rhg(\cR,-)$ with sifted colimits requires imposing homological finiteness properties.

\subsubsection{Homological finiteness in the solid setting}
We are naturally led to one of the first investigations of homological finiteness properties in condensed mathematics. The first is cohomological dimension, the minimal length of a projective resolution of the trivial module. We denote by $\mathrm{cd}_{{\cR}}(G)$ the cohomological dimension of $G$ in the category of solid $\rg$-modules, and by $\mathrm{cd}_{\ring}(G)$ the cohomological dimension of $G$ in the category of profinite $\ring\llbracket G \rrbracket$-modules. We prove that cohomological dimension carries across condensation:
\begin{theorem}[= Theorem~\ref{thm:cd}] If $G$ is a profinite group and $\ring$ a commutative profinite ring, then $\mathrm{cd}_{\ring} G =\mathrm{cd}_{\cR} G$.
\end{theorem}

We also need a homological finiteness property we term type $\mathrm{CP}_{\infty}$. Type $\mathrm{CP}_{\infty}$ is in analogy with and acts as a replacement for the $\mathrm{FP}_{\infty}$-property. Type $\mathrm{FP}_{\infty}$ is inadequate in condensed mathematics because it hinges on the existence of a resolution with finitely generated projectives. Condensed and solid modules are sheaves, and notions such as ``finite'' or ``finitely generated'' are ill-behaved even if one can define them --- they interact poorly with operations such as taking kernels or images. The appropriate notion in the condensed setting is that of a compact object: an object $X$ such that $\mathrm{Hom}(X,-)$ commutes with direct limits. 

\begin{definition}[Type $\mathrm{CP}_n$, $\mathrm{CP}_{\infty}$ and $\mathrm{CP}$]
A group is of type $\mathrm{CP}_{n}$ over $\ring$ if there is a projective resolution $P_{\bullet} \rightarrow \ring$ of modules over the group ring such that $P_i$ is compact for $i\leq n$. A group that is of type $\mathrm{CP}_n$ for all $n$ is said to be of type $\mathrm{CP}_\infty$. If the resolution can additionally be taken to be bounded, then the group is of type $\mathrm{CP}$.\footnote{In the language of derived geometry, type $\mathrm{CP}_{\infty}$ over $\ring$ is the condition that $\ring$ is an almost perfect object in the derived category of modules over the group ring, $\mathrm{CP}$ means that $\ring$ is perfect.}
\end{definition}

For discrete groups, type $\mathrm{FP}_{n}$ is equivalent to type $\mathrm{CP}_{n}$. Namely, a module over a discrete ring is a compact object if and only if it is finitely presented~\cite[pp.~140--141]{kas06}. This is not true for modules over profinite rings. The completed group ring $\widehat{\mathbb{Z}}\llbracket G \rrbracket$ of any infinitely generated profinite group $G$ is finitely presented as a module over itself, but not a compact object (Proposition~\ref{prop:solidandcompact}). Even so, discrete and profinite groups that are not of type $\mathrm{FP}_{\infty}$ are abundant (cf.~\cite[Examples~6.3]{bes97} and~\cite[Proposition~4.6]{coo16a}). This explains the significance of the following theorem:

\begin{theorem}[= Theorem~\ref{thm:cpoo}] If $G$ is a profinite group and $\ring$ a commutative profinite ring, then $\cG$ is of type $\mathrm{CP}_{\infty}$ over $\cR$ for solid modules.
\end{theorem}

This theorem is one of the key advantages of condensed mathematics over existing classical methods: it implies that solid group cohomology $H^i_{\cR}(G,-)$ of any profinite group $G$ commutes with direct limits of coefficient modules. When $G$ has finite cohomological dimension, this lifts to the derived functor $\rhg(\cR, -)$. We shall use the latter in establishing the core isomorphism from Equation~\eqref{eq:key iso}.

Another finiteness property that plays an important role in our paper is condition (FF): cohomology with finite coefficient modules is finite. This property implies (see the proof of Lemma~\ref{lem:corelemma}) that the dualising complex $D^\bullet_{\cR}$ is an inverse limit of finite modules. The solid tensor product with finite modules commutes with arbitrary products, and limits commute with limits, so this means that $D^\bullet_\cR\stg-$ commutes with arbitrary products. This will be used for Equation~\eqref{eq:key iso}, 

\subsubsection{The Tate spectral sequences}

We can now sketch a proof of our core technical result, Lemma~\ref{lem:corelemma}. This says that if $G$ is a profinite group of finite cohomological dimension satisfying condition (FF), the natural transformation from Equation~\eqref{eq:key iso} between the functors $D^\bullet_\cR \stg (-) $ and $ \rhg(\cR,-)$ on $\SolidLG$ is an isomorphism. As explained in the preceding sections, these hypotheses guarantee that both these functors commute with sifted colimits and arbitrary products. The functors agree on the free rank-one module $\rg$. The fact that the compact projectives $\prod_I \rg$ generate $\SolidLG$ under sifted colimits then gives the desired isomorphism. 

Combining this with the identity $\cR \stg \left(D^\bullet_\cR \st (-)\right) \cong D^\bullet_\cR\stg (-)$ from Lemma~\ref{lem:tensorinteract} gives Equation~\eqref{eq:tate iso}. When restricted to condensates of modules over the completed group ring $\ring\llbracket G \rrbracket$, we find, using Equation~\eqref{eq:derived functors comparison}, a natural isomorphism
$$
    \ring \widehat{\otimes}_{\ring \llbracket G \rrbracket} \left((-) \widehat{\otimes}_{\ring} D_{\ring}^{\bullet} \right) \xrightarrow{\cong} \mathrm{RHom}_{\ring \llbracket G \rrbracket}(\ring, -).
$$
Associated with these isomorphisms are spectral sequences 
\begin{align}\begin{split}\label{eq:spectral sequences}
    E_2^{pq}= &\mathbf{H}^{\cR}_{-p}(G; H^q(M \st D^\bullet_{\cR})) \Rightarrow \mathbf{H}_{\cR}^{p+q}(G;M) \\
   E_2^{pq}=&H^\ring_{-p}(G; H^q(M \widehat{\otimes}_\ring D^\bullet_\ring)) \Rightarrow H_\ring^{p+q}(G;M).    
\end{split}
\end{align}
that we refer to as Tate spectral sequences in analogy with the classical theory. From these, we can deduce our main results. 

We note that our proof relies on the facts that $\SolidLG$ is generated by compact projectives under sifted colimits and that every profinite group is of type $\mathrm{CP}_{\infty}$ over $\cR$. This explains why this spectral sequence has not been established for profinite modules previously. 

\subsection{Main results}

We prove two versions of duality for profinite groups, one for solid modules and one for continuous modules. Both state that a profinite group $G$ satisfying condition (FF) is a duality group if and only if it is Cohen--Macaulay. This condition goes back to Verdier \cite[Définition~4.1]{Verdier1965}. 

\begin{definition}\label{def:cmclassical}\label{def:cm}
Let $G$ be a profinite group and $\ring$ a commutative profinite ring such that $G$ satisfies condition (FF). 
\begin{itemize}
    \item We call $G$ a \emph{profinite Cohen--Macaulay group over $\ring$ of dimension $n$} if $\cd_{\ring}(G) = n$, $H_{\ring}^{\bullet}(G, \ring \llbracket G \rrbracket)$ is concentrated in degree $n$ and is flat as a profinite $\ring$-module.
    \item Analogously, $G$ is \emph{solid Cohen--Macaulay group over $\cR$ of dimension $n$} if $\cd_{\cR}(G) = n$, $\mathbf{H}_{\cR}^{\bullet}(G, \rg)$ is concentrated in degree $n$ and is flat as a solid $\cR$-module.
\end{itemize}
\end{definition}

The first main theorem concerns duality for solid modules. We hereby recall that condensates of profinite and discrete modules are solid. Thus, this theorem also establishes duality for profinite and discrete modules uniformly within condensed mathematics whereas they previously had to be treated differently. 

\begin{thmA}[Solid duality]\label{thmA}
Let $G$ be a profinite group and $\ring$ a commutative profinite ring such that $\cd_{\cR}(G) = n$ and $G$ satisfies condition (FF). Then $G$ is a solid duality group over $\cR$ of dimension $n$, meaning that there is a solid dualising module $D_{\cR}$ and that for every solid $\rg$-module $M$ and every integer $k$ there exists an isomorphism
\[ \mathbf{H}_k^{\cR}(G, M \st D_{\cR}) \xrightarrow{\cong} \mathbf{H}_{\cR}^{n-k}(G, M)  \, , \]
if and only if $G$ is a solid Cohen--Macaulay group over $\cR$.
\end{thmA}

From similar considerations we get an analogous result purely in the profinite setting. This theorem is classical for specific, well-behaved profinite rings (cf.\! \textsection\ref{ssec:background}), but is new for \emph{arbitrary} profinite coefficient rings. In the classical setting, one uses various finiteness and topological properties of the group rings to gain control over the derived functors on the category of profinite modules. In condensed mathematics, the category of profinite modules is fully faithfully embedded in that of solid modules, which has a very well-behaved derived category. This result is one of the first applications of condensed mathematics that has not been achieved in this uniformity or generality through any other means:

\begin{thmA}[Continuous duality]\label{thmB}
Let $G$ be a profinite group and $\ring$ a commutative profinite ring such that $\cd_{\ring}(G) = n$ and $G$ satisfies condition (FF). Then $G$ is a profinite duality group over $\ring$ of dimension $n$, meaning that there is a profinite dualising module $D_{\ring}$ and that for every profinite $\ring \llbracket G \rrbracket$-module $M$ and every integer $k$ there exists an isomorphism
\begin{equation}\label{eq:profinduality}
H_k^{\ring}(G, M \widehat{\otimes}_{\ring} D_{\ring}) \rightarrow H_{\ring}^{n-k}(G, M) \, , \tag{$\ast$}
\end{equation}
if only if $G$ is a profinite Cohen--Macaulay group over $\ring$. 

If $G$ is a profinite duality group of dimension $n$ and of type $\mathrm{FP}_{\infty}$ over $\ring$ with finitely generated $D_{\ring}$, then duality of $G$ also holds for discrete $\ring \llbracket G \rrbracket$-modules, meaning that Equation~(\ref{eq:profinduality}) holds for any discrete $\ring \llbracket G \rrbracket$-module $M$.
\end{thmA}

The requirement that $G$ is of type $\mathrm{FP}_{\infty}$ over $\ring$ for discrete coefficients ensures that continuous group homology is well defined (cf.~\cite[p.~378]{sym00}).

The Cohen--Macaulay condition in the solid setting is equivalent to that in the profinite setting: continuous group (co)homology is isomorphic to solid group (co)homology via the condensation functor (Theorem~\ref{thm:profintosolid}), we have $\mathrm{cd}_{\ring}(G) = \mathrm{cd}_{\cR}(G)$ (Theorem~\ref{thm:cd}), and $\mathbf{H}_{\cR}^n(G, \rg)$ is a flat solid $\cR$-module if and only if $H_{\ring}^n(G, \ring \llbracket G \rrbracket)$ is a flat profinite $\ring$-module (Lemma~\ref{lem:flatness}). Theorem~\ref{thmA} and Theorem~\ref{thmB} thus imply:

\begin{corA}\label{corC}
Let $G$ be a profinite group and $\ring$ a commutative profinite ring such that $\cd_{\ring}(G) = n$ and $G$ satisfies condition (FF). Then $G$ is a solid duality group over $\cR$ if and only if it is a profinite duality group over $\ring$. 
\end{corA}

\subsubsection{Proof sketch of the main theorems}
In Theorem~\ref{thmA} and Theorem~\ref{thmB} the Cohen--Macaulay condition on $G$ implies that spectral sequences from Equation~\eqref{eq:spectral sequences} collapse and that $G$ is a duality group. For the other implication where $G$ is a duality group, one uses that $\mathbf{H}_{\bullet}^{\cR}(G, \rg \st M)$ (respectively, $H_{\bullet}^{\ring}(G, \ring \llbracket G \rrbracket \widehat{\otimes}_{\ring} M)$) is concentrated in a single degree (Lemma~\ref{lem:Shapiro} and~\cite[Lemma~3.3.4]{sym00}). This implies first that $D_{\cR} = \mathbf{H}_{\cR}^n(G, \rg)$ and $D_{\ring} = H_{\ring}^n(G, \ring \llbracket G \rrbracket)$ are flat and then that $G$ is Cohen--Macaulay.

\subsection{Section summary}
In Section~\ref{sec:condsetting} we introduce all the relevant background on condensed mathematics. Importantly, we prove that all homological algebra for profinite groups with continuos modules can be equivalently performed with solid modules. The investigation of homological finiteness conditions in condensed mathematics is the subject of Section~\ref{sec:finitenessconds}. Section~\ref{sec:varia} proves preliminary results necessary to demonstrate the main theorems. Then Section~\ref{sec:mains} proves the Tate spectral sequence as well as the main theorems. It also generalises Serre--Verdier--Tate duality as well as Bieri--Eckmann duality. Lastly, Section~\ref{sec:exls} constructs novel examples and establishes group-theoretic properties of duality groups.

\subsection*{Acknowledgements}

We would like to express our gratitude to Thomas Weigel for sharing his knowledge with us. We express our thanks to Emma Brink for the insight she has provided into her article on condensed group cohomology. We are indebted to Peter Kropholler who drew our attention to homological finiteness properties in condensed mathematics.

MG wishes to thank Ilya Kazachkov and Montse Casals Ruiz for having hosted him as a research visitor at the University of the Basque Country and gratefully acknowledges support by DFG grant KL 2162/4-1, within a joint AEI-DFG-funded project. TW gratefully acknowledges support from the Royal Society through Richard Wade's University Research Fellowship, the European Commission through a Marie Sk{\l}odowska-Curie Individual Fellowship with reference number 101203556 – DisQ FA, and from the University of Oxford and Aalto University.

\section{The condensed setting}\label{sec:condsetting}

This section provides the necessary background on condensed mathematics and demonstrates that homological algebra for profinite groups with continuous modules can be equivalently performed with solid modules. The goal of the section is to establish Theorem~\ref{thm:profintosolid} stating that the derived tensor product on the profinite and the solid side are isomorphic via the condensation operation where the same pertains to the derived Hom-functor. We do not claim much originality for this result, part of it can be extracted from \cite{bri25}, and most of the statement is implicit in \cite{tang24}.

\subsection{Basics of condensed modules}

\subsubsection{Condensed sets}
The building blocks of condensed mathematics are \emph{condensed sets}. We define these following Clausen and Scholze's account~\cite[Lecture~I and~II]{sch22}. Further below we provide a plain and straightforward definition of $\kappa$-condensed sets after introducing the relevant Grothendieck sites of our sheaves. Naïvely, condensed sets are sheaves on the category of profinite spaces $\mathbf{Pro}$, equipped with the Grothendieck topology whose coverings are finite collections of (continuous) maps $\lbrace f_i\colon S_i \rightarrow S \rbrace_{i = 1}^n$ such that the induced map $\bigsqcup_{i = 1}^n S_i \rightarrow S$ is surjective. This naïve definition leads to size issues: $\mathbf{Pro}$  is not essentially small. One gets around these size issues by taking for each uncountable strong limit cardinal $\kappa$ the category $\mathrm{Cond}_{\kappa}(\mathbf{Set})$ of sheaves on the site of profinite spaces with at most $\kappa$ clopen subsets. In more down-to-earth terms, a (contravariant) functor
\[ X: \mathbf{Pro}_{\kappa}^{\mathrm{op}} \rightarrow \mathbf{Set} \]
is a $\kappa$-condensed space if the following two conditions are satisfied.
\begin{enumerate}
    \item For every finite collection $\lbrace S_i \rbrace_{i = 1}^n$ of objects in the category $\mathbf{Pro}_{\kappa}$ there is a bijection $X(\bigsqcup_{i = 1}^n S_i) \rightarrow \prod_{i = 1}^n X(S_i)$.
    \item Let $f: T \rightarrow S$ be a (continuous) surjection in $\mathbf{Pro}_{\kappa}$ and denote the projections from the fiber product by $p_1, p_2: T \times_S T \rightarrow T$. Then $X(f): X(S) \rightarrow X(T)$ is injective with image
\[ \mathrm{Im}(X(f)) = \lbrace x \in X(T) \mid X(p_1)(x) = X(p_2)(x) \in X(T \times_S T) \rbrace \, . \]
\end{enumerate}
The category of condensed sets is defined by taking the colimit
\[ \mathrm{Cond}(\mathbf{Set}) := \varinjlim_{\kappa \in K} \mathrm{Cond}_{\kappa}(\mathbf{Set}) \]
over all such cardinals. These size issues are mostly left to the background for our purposes. For a more thorough definition we refer the reader to~\cite[Chapter~2]{bri25}.

One can think of a condensed set $X\in \mathrm{Cond}(\mathbf{Set})$ as encoding a topological space through mapping spaces $X(S)$ of ``continuous maps $S \rightarrow X$'' from a profinite space $S$. This is literally true for those condensed sets that are in the image of the \emph{condensation functor} $\underline{(-)}\colon T1\tn{-}\mathbf{Top}\to \mathrm{Cond}(\mathbf{Set})$ that takes a $T1$ topological space $Y$ to the sheaf $S\mapsto \mathrm{Hom}_{\tn{cont}}(S,Y)$. By the Yoneda lemma, condensation is fully faithful when restricted to the category of profinite spaces; it is faithful on the entirety of $ T1\tn{-}\mathbf{Top}$.

\subsubsection{Condensed algebraic structures}

All concrete (defined as structure on a set) mathematical objects have condensed counterparts: a condensed group, ring, module, etc. is a sheaf of groups, rings, modules, etc. on the site $\mathbf{Pro}$, up to the size issues alluded to above. Condensates of $T1$ topological groups (resp.\! rings, modules etc) are condensed groups (resp.\! rings, modules etc)~\cite[p.~8]{Scholze2019}.

Condensed modules have  excellent category-theoretic properties \cite[Theorem~1.10]{Scholze2019}. Denote by $\mathrm{Cond}(\mathbf{Ab})^{\heartsuit}$ the abelian category of condensed abelian groups and by $\mathrm{Cond}(\mathbf{Mod}_{\mathcal{R}})^{\heartsuit}$ the abelian category of condensed modules over a condensed ring $\mathcal{R}$. In these categories, all limits and colimits exist. Arbitrary products, arbitrary direct sums and filtered colimits are exact. They are generated by compact projective objects, meaning that every condensed module admits a surjective homomorphism from a coproduct of these compact projective objects. In particular, the category $\mathrm{Cond}(\mathbf{Mod}_{\mathcal{R}})^{\heartsuit}$ has enough projectives. Below, we give a set of compact projective generators.

\begin{notation}
We will denote by $\mathrm{Cond}(\mathbf{Ab})$ the derived category of $\mathrm{Cond}(\mathbf{Ab})^{\heartsuit}$ and by $\mathrm{Cond}(\mathbf{Mod}_{\mathcal{R}})$ the derived category of $\mathrm{Cond}(\mathbf{Mod}_{\mathcal{R}})^{\heartsuit}$ for a condensed ring $\mathcal{R}$.
\end{notation}

\begin{warning}[{\cite[Warning~2.18]{Scholze2019}}]
Although the derived category $D(\mathrm{Cond}(\mathbf{Ab})^{\heartsuit})$ of condensed abelian groups might not be equivalent to $\mathrm{Cond}(D(\mathbf{Ab}^{\heartsuit}))$, the $\infty$-category $\mathcal{D}(\mathrm{Cond}(\mathbf{Ab})^{\heartsuit})$ is nevertheless equivalent to $\mathrm{Cond}(\mathcal{D}(\mathbf{Ab}^{\heartsuit}))$.
\end{warning}

\subsubsection{Free condensed modules}
Free condensed modules play an important role in our investigation, especially because we can find a class of compact projective generators among them. According to the proof of Theorem~A.15 in~\cite{tang24}, there is a free condensed $\mathcal{R}$-module $\mathcal{R}[X]$ over any condensed set $X$. Taking free condensed $\mathcal{R}$-modules is left adjoint to the forgetful functor $\mathrm{Cond}(\mathbf{Mod}_{\mathcal{R}})^{\heartsuit} \rar \mathrm{Cond}(\mathbf{Set})$. Regarding compact projective generators, a topological space $S$ is called extremally disconnected if the closure of any open subspace is again open~\cite[Definition~1.1]{gle58}. These are exactly the projective profinite spaces~\cite[Theorem~1.2]{gle58} and the condensate of any extremally disconnected space $\underline{S}$ is a projective condensed set~\cite[p.~16]{sch22}. The free condensed $\mathcal{R}$-modules $\mathcal{R}[\underline{S}]$ where $S$ ranges through extremally disconnected spaces are compact projective generators for $\mathrm{Cond}(\mathbf{Mod}_{\mathcal{R}})^{\heartsuit}$~\cite[p.12--13]{Scholze2019}. This means that for every condensed $\mathcal{R}$-module $M$ there is a collection of extremally disconnected spaces $S_i$ such that there is an epimorphism $\bigoplus_{i \in I} \mathcal{R}[\underline{S_i}] \rightarrow M$.

\begin{remark}\label{rem:projectives}
Not every free condensed module is projective. For instance, if $I$, $J$ are infinite discrete spaces and ${\beta}I$, ${\beta}J$ denote their Stone--Čech compactifications, then the free condensed $\underline{\Z}$-module $\underline{\Z}[\underline{{\beta}I \times {\beta}J}]$ is not projective by~\cite[Proposition~3.7]{sch22}. In general, it is only true that free condensed modules $\mathcal{R}[\underline{S}]$ are projective for extremally disconnected spaces $S$. The reason for this is that left adjoints, such as taking free condensed modules, preserve projectives if the right adjoint is exact.
\end{remark}

\subsubsection{Tensors and Homs}\label{sssec:tensorhoms}

Condensed modules carry analogues of the usual tensor products of modules over rings. We denote the (derived) tensor product over a condensed ring $\mathcal{R}$ by $\otimes_\mathcal{R}$. We denote the derived $\mathcal{R}$-linear Hom-functor by $\mathrm{RHom}_{\mathcal{R}}(-, -)$. When $\mathcal{R}$ is commutative, there is an internal Hom 
$\mathrm{R}\underline{\mathrm{Hom}}_{\mathcal{R}}(A, B) \in \mathrm{Cond}(\mathbf{Mod}_\mathcal{R})$ between $A,B\in \mathrm{Cond}(\mathbf{Mod}_\mathcal{R})$, satisfying the following tensor-hom adjunction for all $A,B,C\in \mathrm{Cond}(\mathbf{Mod}_\mathcal{R})$:

\begin{equation*}
\mathrm{RHom}_{\mathcal{R}}(A\otimes_{\mathcal{R}}B,C) \cong \mathrm{RHom}_{\mathcal{R}}(A,\mathrm{R}\underline{\mathrm{Hom}}_{\mathcal{R}}(B, C)).
\end{equation*}
The counit of this adjunction is the evaluation map
$$
\tn{ev}_A\colon A\otimes_{\mathcal{R}}\mathrm{R}\underline{\mathrm{Hom}}_{\mathcal{R}}(A, B) \to B.
$$
The composition of internal homs is defined to be the $\mathcal{R}$-linear map
\begin{equation}\label{eq:internalcomposition}
\underline{\circ}\colon\mathrm{R}\underline{\mathrm{Hom}}_{\mathcal{R}}(A, B)\otimes_\mathcal{ R}\mathrm{R}\underline{\mathrm{Hom}}_{\mathcal{R}}(B, C)\to \mathrm{R}\underline{\mathrm{Hom}}_{\mathcal{R}}(A, C)    
\end{equation}
adjoint to the composite of evaluation maps
$$
A\otimes_{\mathcal{R}}\mathrm{R}\underline{\mathrm{Hom}}_{\mathcal{R}}(A, B)\otimes_\mathcal{ R}\mathrm{R}\underline{\mathrm{Hom}}_{\mathcal{R}}(B, C)\xrightarrow{\tn{ev}_A} B\otimes_{\mathcal{R}}\mathrm{R}\underline{\mathrm{Hom}}_{\mathcal{R}}(B, C)\xrightarrow{\tn{ev}_B} C. 
$$

As group cohomology is of concern, we consider condensed group rings, their tensor products and Hom-functors. If $\mathcal{G}$ denotes a condensed group, then there exists a condensed group ring $\mathcal{R}[\mathcal{G}]$ by~\cite[p.~14]{tang24}. By the same source, the category $\mathrm{Cond}(Mod_{\mathcal{R}[\mathcal{G}]})$ is equivalent to the category of derived $\mathcal{R}$-modules with a commuting $\mathcal{G}$-action. Importantly for group cohomology, the coefficient ring $\mathcal{R}$ is a condensed $\mathcal{R}[\mathcal{G}]$-module.

Analogous to the classical setting, we have both $\otimes_\mathcal{R}$ and $\otimes_{\mathcal{R}[\mathcal{G}]}$ at our disposal, as well as $\mathrm{R}\underline{\mathrm{Hom}}_{\mathcal{R}}(-, -)$ and $\mathrm{R}\underline{\mathrm{Hom}}_{\mathcal{R}[\mathcal{G}]}(-, -)$, where we emphasise that for the last, the underline signals that $\mathrm{R}\underline{\mathrm{Hom}}_{\mathcal{R}[\mathcal{G}]}(M, N)$ is a condensed $\mathcal{R}$-module for any pair of $\mathcal{R}[\mathcal{G}]$-modules $M$ and $N$. The $\mathcal{R}[\mathcal{G}]$-tensor product and hom satisfy the usual tensor-hom adjunction for the tensor product of right and left modules.

The inverse map gives an anti-isomorphism $\mathcal{R}[G]\cong (\mathcal{R}[\mathcal{G}])^\tn{op}$, and hence an equivalence $\mathcal{I}$ between the categories of left and right $\mathcal{R}[\mathcal{G}]$-modules. Note that
\begin{align*}
    \mathrm{R}\underline{\mathrm{Hom}}_{\mathcal{R}[\mathcal{G}]}(\mathcal{I}(M),A)&\cong \mathcal{I}(\mathrm{R}\underline{\mathrm{Hom}}_{\mathcal{R}[\mathcal{G}]}(M,A))\\
    \mathcal{I}(M)\otimes_{\mathcal{R}[\mathcal{G}]} N& \cong \mathcal{I}(N)\otimes_{\mathcal{R}[\mathcal{G}]} M
\end{align*}
where $M$ can be taken to be either a left or right $\mathcal{R}[\mathcal{G}]$-module in the first equation and $M$ and $N$ are both left $\mathcal{R}[\mathcal{G}]$-modules in the second. Following the usual conventions in the literature \cite[Chapter III.0]{bro82}, we will suppress $\mathcal{I}$ from the notation for the rest of this paper, and take all our modules to be left modules, tacitly using $\mathcal{I}$ when needed. With this convention, we have hom-tensor adjunctions
\begin{align*}
    \mathrm{R}\underline{\mathrm{Hom}}_{\mathcal{R}}(M\otimes_{\mathcal{R}[\mathcal{G}]} N, A)&\cong \mathrm{R}\underline{\mathrm{Hom}}_{\mathcal{R}[\mathcal{G}]}(M, \mathrm{R}\underline{\mathrm{Hom}}_{\mathcal{R}}( N, A)) \quad \text{and} \\
    \mathrm{R}\underline{\mathrm{Hom}}_{\mathcal{R}}(M\otimes_{\mathcal{R}[\mathcal{G}]} N, A)&\cong \mathrm{R}\underline{\mathrm{Hom}}_{\mathcal{R}[\mathcal{G}]}(N, \mathrm{R}\underline{\mathrm{Hom}}_{\mathcal{R}}( M, A)),
\end{align*}
where $M$ and $N$ are both left $\mathcal{R}[\mathcal{G}]$-modules, and $A$ is an $\mathcal{R}$-module. 

The tensor product $\otimes_\mathcal{R}$ is again a $\mathcal{R}[\mathcal{G}]$-module (equipped with the diagonal action), and $\mathrm{R}\underline{\mathrm{Hom}}_{\mathcal{R}}(-, -)$ acts the internal hom with respect to this symmetric monoidal structure:
$$
\mathrm{R}\underline{\mathrm{Hom}}_{\mathcal{R}[\mathcal{G}]}(M\otimes_{\mathcal{R}} N, L)\cong \mathrm{R}\underline{\mathrm{Hom}}_{\mathcal{R}[\mathcal{G}]}(M, \mathrm{R}\underline{\mathrm{Hom}}_{\mathcal{R}}( N, L)),
$$
for $L,M,N$ left $\mathcal{R}[\mathcal{G}]$-modules.
\subsection{Basics of solid modules}

Solid modules are essential to Theorem~\ref{thm:profintosolid} which ensures that homological algebra for profinite groups via continuous modules can be performed equivalently in condensed mathematics. In this subsection, we present the basics of these condensed modules satisfying a form of completion.

\subsubsection{Definition of solid modules}
The definition of solid modules comes in two steps where we first define a pre-analytic ring.

\begin{definition}
[{\cite[Definition~7.1]{Scholze2019}}] A pre-analytic ring $(\mathcal{A}, S \mapsto \mathcal{A}[S]^{\square})$ is a condensed ring $\mathcal{A}$ together with a functor
\begin{align*}
c_{\mathcal{A}}\colon \lbrace \text{extremally disconnected spaces} \rbrace &\rar \mathrm{Cond}(Mod_{\mathcal{A}}^{\heartsuit}), \\
S &\mapsto \mathcal{A}[S]^{\square}
\end{align*}
taking finite disjoint unions to products and with a natural transformation
\[ C_{\mathcal{A}}\colon \mathcal{A}[-] \rar c_{\mathcal{A}}, \quad C_{\mathcal{A}}(S)\colon \mathcal{A}[\underline{S}] \rar \mathcal{A}[S]^{\square} \]
where $\mathcal{A}[\underline{S}]$ denotes the free condensed $\mathcal{A}$-module on $\underline{S}$.
\end{definition}

Thus, a pre-analytic structure on a condensed ring $\mathcal{A}$ is a systematic assignment of condensed $\mathcal{A}$-modules $\mathcal{A}[S]^{\square}$ that we can think of as being free. However, this assignment is not defined for any profinite space $S$ a priori. In case of an analytic ring one can extend this choice of ``free'' condensed modules to a full subcategory of $\mathcal{A}$-modules that we call solid modules.

\begin{definition}
[{\cite[Definition~7.4]{Scholze2019})}] An analytic ring $(\mathcal{A}, S \mapsto \mathcal{A}[S]^{\square})$ is a pre-analytic ring such that for any complex of condensed $\mathcal{A}$-modules
\[ C: {} \dots {} \rar C_2 \rar C_1 \rar C_0 \rar 0 \]
where all $C_i$ are direct sums of objects $\mathcal{A}[T]^{\square}$ for varying extremally disconnected $T$, the homomorphism of condensed abelian groups
\[ \mathrm{R}\underline{\mathrm{Hom}}_{\mathcal{A}}(C_{\mathcal{A}}, C)\colon \mathrm{R}\underline{\mathrm{Hom}}_{\mathcal{A}}(\mathcal{A}[S]^{\square}, C) \rar \mathrm{R}\underline{\mathrm{Hom}}_{\mathcal{A}}(\mathcal{A}[\underline{S}], C) \]
is an isomorphism for every extremally disconnected space $S$.
\end{definition}

The promised well-behaved category of solid modules is given in the following lemma.

\begin{lemma}\label{lem:analyticrings}
[{\cite[Proposition~7.5]{Scholze2019}}] Let $(\mathcal{A}, S \rar \mathcal{A}[S]^{\square})$ be an analytic ring. Then the full subcategory
\[ \Solid_{\mathcal{A}}^{\heartsuit} \subseteq \mathrm{Cond}(Mod_{\mathcal{A}})^{\heartsuit} \]
of all $\mathcal{A}$-modules $M$, for which for every extremally disconnected space $S$ the homomorphism
\[ \mathrm{Hom}_{\mathcal{A}}(C_{\mathcal{A}}(S), M) \colon \mathrm{Hom}_{\mathcal{A}}(\mathcal{A}[S]^{\square}, M) \rar \mathrm{Hom}_{\mathcal{A}}(\mathcal{A}[\underline{S}], M) = M(S) \]
is an isomorphism, is an abelian category stable under all limits and colimits. The derived internal Hom-functors $\mathrm{R}\underline{\mathrm{Hom}}_{\mathcal{A}}$ agree. There is a left adjoint to the above forgetful functor
\[ (-)^{{\square}\mathcal{A}}\colon \mathrm{Cond}(Mod_{\mathcal{A}}^{\heartsuit}) \rar \Solid_{\mathcal{A}}^{\heartsuit}, \: M \mapsto M^{\square} \]
that is the unique colimit-preserving extension of $C_{\mathcal{A}}(S)\colon \mathcal{A}[\underline{S}] \rar \mathcal{A}[S]^{\square}$. The objects $\mathcal{A}[S]^{\square}$ for $S$ extremally disconnected are compact projective generators of $\Solid_{\mathcal{A}}^{\heartsuit}$.
\end{lemma}

\begin{definition}
In the notation of Lemma~\ref{lem:analyticrings}, the objects in $\Solid_{\mathcal{A}}$ are called \emph{solid} $\mathcal{A}$-modules and the functor $(-)^{{\square}\mathcal{A}}\colon \mathrm{Cond}(Mod_{\mathcal{A}}) \rar \Solid_{\mathcal{A}}$ is called \emph{solidification}.
\end{definition}

As $\Solid_{\mathcal{A}}^{\heartsuit}$ is generated by compact projective objects, it has enough projectives. We will make frequent use of the fact that the derived category of any abelian category with enough projectives is generated under {sifted colimits} by projective objects:

\begin{proposition}[{\cite[Proposition 12.4]{scholze2026}}]\label{lem:siftedcolimits}
Let $\mathcal{A}$ be an analytic ring. Then every object $M$ in the derived category of solid $\mathcal{A}$-modules can be written as a sifted colimit of chain complexes of compact projective solid $\mathcal{A}$-modules.
\end{proposition}

\subsubsection{Examples and their properties}
Recall for a profinite group $G = \varprojlim_{i \in I} G_i$ and a profinite ring $R = \varprojlim_{j \in J} R_j$ that the completed group ring $R\llbracket G \rrbracket = \varprojlim_{(i, j) \in I \times J} R_j[G_i]$ is the profinite completion of the ``ordinary'' group ring $R[G]$.

For an extremally disconnected space $S = \varprojlim_{i \in I} S_i$ define the condensed $\underline{\Z}$-module $\Z[S]^{\solid} := \varprojlim_{i \in I} \underline{\Z}[\underline{S_i}]$. Then $(\underline{\Z}, S \rar \Z[S]^{\solid})$ is an analytic ring \cite[Theorem~5.8]{Scholze2019}.

\begin{definition}
Solid $\underline{\Z}$-modules are called \emph{solid abelian groups}. The category of solid abelian groups is denoted by $\Solid_{\mathbf{Ab}}$ and solidification by $(-)^{\solid}\colon \mathrm{Cond}(\mathbf{Ab}) \rar \Solid_{\mathbf{Ab}}$.
\end{definition}

This analytic ring structure on the integers $\underline{\Z}$ can be extended to any condensed ring.

\begin{lemma}\label{lem:analyticstructure}
(\cite[Proposition~3.6 and Lemma~3.13]{tang24}) Let $\mathcal{R}$ be a condensed ring. Then the following assertions hold.
\begin{enumerate}
\item $(\mathcal{R}, S \mapsto \mathcal{R}[\underline{S}]^{\solid})$ is an analytic ring.
\item If $\mathcal{R}$ is solid as a condensed abelian group, then a condensed $\mathcal{R}$-module is solid if and only if its underlying condensed abelian group is solid.
\item There is an equivalence of categories $\Solid_{\mathcal{R}} = \Solid_{\mathcal{R}^{\solid}}$.
\item For any $M, N \in  \Solid_{\mathcal{R}}$ the internal Hom-functor $\mathrm{R\underline{Hom}}_{\mathcal{R}}(M, N)$ lies in $\Solid_{\mathcal{R}}$.
\end{enumerate}
\end{lemma}

Important to us, profinite rings possess particularly well-behaved solid modules. We remind the reader that not all free condensed modules are projective (cf. Remark~\ref{rem:projectives}).

\begin{lemma}\label{lem:coeffringanalytic}
Let $R$ be a profinite ring or the integers $\Z$. If one takes the analytic rings structure $(\underline{R}, S \mapsto \underline{R}[\underline{S}]^{\solid})$ from Lemma~\ref{lem:analyticstructure}, then the following assertions hold.
\begin{enumerate}
\item $\underline{R}^{\solid} = \underline{R}$.
\item The compact projective generators of $\Solid_{\underline{R}}^{\heartsuit}$ are of the form $\prod_I \underline{R}$ where $I$ is any set.
\item If $S$ is a profinite space, then $\underline{R}[\underline{S}]^{\solid}$ is a compact projective solid $\underline{R}$-module.
\end{enumerate}
\end{lemma}

\begin{proof}
The first assertion is due to~\cite[Lemma~3.9]{tang24} and the fact that $\underline{R} = \underline{R}[\ast]^{\solid}$ is a projective solid $\underline{R}$-module. The second assertion holds by~\cite[Proposition~3.10]{tang24}. The third assertion follows from the proof of~\cite[Corollary~5.5]{Scholze2019} and \cite[Equation~1, p.~10]{tang24}.
\end{proof}

\begin{notation}
For a profinite ring $R$ and a profinite space $S$ we write $R[S]^{\solid}$ for $\underline{R}[\underline{S}]^{\solid}$.
\end{notation}

\subsection{Solid modules and (co)homology of profinite groups}

The purpose of this subsection is to prove the section's main theorem, namely Theorem~\ref{thm:profintosolid}. The following lemma due to Tang yields the necessary link from continuous modules over a profinite ring to solid modules. This link allows one to perform all continuous homological algebra for profinite groups equivalently in condensed mathematics.

\begin{lemma}\label{lem:profinitembeddings}
Let $R$ be a profinite ring.
\begin{enumerate}
\item The fully faithful condensation functor $\underline{-}\colon T1\text{-}Mod_R^{\heartsuit} \rar \mathrm{Cond}(\mathbf{Mod}_{\underline{R}})^{\heartsuit}$ takes discrete and profinite $R$-modules to solid $\underline{R}$-modules. Restricted to profinite $R$-modules, $\underline{-}$ is exact where the same holds true if one restricts to discrete modules.
\item If $S$ is a profinite space, then $\underline{R\llbracket S \rrbracket} = R[S]^{\solid}$. In particular, the condensation functor $\underline{-}$ restricted to profinite $R$-modules preserves projectives.
\end{enumerate}
\end{lemma}
\begin{proof}
The first assertion is due to \cite[Theorem~3.2]{tang24} and its proof. The second assertion follows from the statement and the proof of~\cite[Theorem~3.14]{tang24}.
\end{proof}

We also need to construct a solid group ring and solid modules over them in order to define solid group homology and cohomology. By the previous lemma, the condensate of the completed group ring agrees with the solidification of condensed group ring. We thus introduce the notation
$$
\rg := \ring[G]^{\solid} =\underline{\ring\llbracket G \rrbracket},
$$
and refer to $\rg$ as the \emph{solid group ring}.

\begin{lemma}\label{lem:analyticgrouprings}
Let $\ring$ be a commutative profinite ring and $G$ a profinite group. Equip $\rg$ with the analytic structure from Lemma~\ref{lem:analyticstructure}. Then:
\begin{enumerate}
\item The category $\Solid_{\rg}$ is equivalent to the category of solid ${\cR}$-modules with a commuting $\underline{G}$-action.
\item The category $\SolidLG$ embeds fully faithfully into $\mathrm{Cond}(\mathbf{Mod}_{\cR[\underline{G}]})$. Solidification takes the form $(-)^{\solid}\colon \mathrm{Cond}(\mathbf{Mod}_{\cR[\underline{G}]}) \rar \SolidLG$.
\end{enumerate}
\end{lemma}
\begin{proof}
The first assertion follows from~\cite[pp.~23--24]{tang24}. The second assertion follows from the first and from the second assertion of Lemma~\ref{lem:analyticstructure}.
\end{proof}

We take $\st$ on $\LSolid$ to be the solidification of the tensor product $\otimes_{\cR}$ on $\mathrm{Cond}(\mathbf{Mod}_{\cR})$. By the above lemma, $\st$ is also defined within $\SolidLG$. This is a symmetric monoidal structure that by construction makes solidification into a symmetric monoidal functor. We similarly have $\stg$ defined by $M \stg N= (M \otimes_{\cR[\underline{G}]} N)^{\solid}$ for $M, N \in \SolidLG$. The internal ${\cR}$-linear Hom on $\mathrm{Cond}(\mathbf{Mod}_{\cR[\underline{G}]})$ restricts to an internal Hom $\rh(-,-)$ for $\SolidLG$, while the $\rg$-linear Hom $\rhg$ takes values in solid ${\cR}$-modules. As a consequence of the adjunctions from \textsection\ref{sssec:tensorhoms}, these functors satisfy for $L,M,N\in \SolidLG$ and $A\in \Solid$
\begin{align}
    \rh(M \stg N, A)& \cong \rhg(N, \rh(M,A))\label{eq:adjunction1}\\
    \rh(M \stg N, A)& \cong \rhg(M, \rh(N,A))\label{eq:adjunction1a}
\end{align}
and
\begin{equation}\label{eq:adjunction3}
\rhg(L \st M, N)\cong \rhg(L, \rh(M,N)).
\end{equation}	
If $M$ is a solid $\rg$-bimodule and $N$ is a solid left $\rg$-module, note that $\rhg(M,N)$ is naturally a solid left $\rg$-module. We have
\begin{equation}\label{eq:adjunction2}
\rhg(M \stg N, L) \cong \rhg(N, \rhg(M,L)),
\end{equation}
the left module structures of $M$ and $L$ are just along for the ride. For similar reasons, for $M,N\in\SolidLG$ and $A\in \LSolid$, the hom-tensor adjunction of ${\cR}$-modules
\begin{equation}\label{eq:adjunction4}
    \rh(M,\rh(N,A))\cong\rh(M\st N, A)
\end{equation}
is an isomorphism of $\rg$-modules.

With the derived functors in place, we can define (co)homology. Besides the solid enriched hom $\rhg(-,-)$, there is also the external hom $\mathrm{RHom}(-,-)$. This means there are two versions of cohomology.

\begin{definition}
Let $G$ be a profinite group and $\ring$ a profinite commutative ring. We define the solid group homology of $G$ with coefficient ring ${\cR}$ as
\[ \mathbf{H}_{\bullet}^{\cR}(G, -) := H_{\bullet}(\cR \otimes_{\rg}^{\solid} -) \, . \]
We the solid internal group cohomology of $G$ with coefficient ring ${\cR}$ as
$$
\mathbf{H}_{\cR}^{\bullet}(G, -) := H^{\bullet}\big(\rhg(\cR, -)\big),
$$
and the solid external group cohomology of $G$ as
$$
H_{\cR}^{\bullet}(G, -) := H^{\bullet}\big(\mathrm{RHom}_{\rg}(\cR, -)\big).
$$
\end{definition}

We consider when cohomology groups of profinite groups are profinite, which is necessary in the third assertion of Theorem~\ref{thm:profintosolid} and ultimately in Theorem~\ref{thmA} and Theorem~\ref{thmB}.

\begin{lemma}\label{lem:profincohomgroups}
Let $G$ be a profinite group and $\ring$ be a commutative profinite ring. Then the following are equivalent.
\begin{enumerate}
    \item $H_{\ring}^k(G, M)$ is finite for every finite $M$ and any $k$.
    \item $H_{\ring}^k(G, M)$ is profinite for every profinite $M$ and any $k$.
\end{enumerate}
\end{lemma}

\begin{proof}
This essentially follows from the statement and proof of Theorem~1.6 in~\cite{Pletch1980}. To prove that the first assertion implies the second, we use that the second variable of Hom-functors commutes with inverse limits. By~\cite[Proposition~2.2.4]{rib10}, inverse limits of profinite groups are exact. Thus, the result follows because group cohomology commutes with inverse limits. For the other implication one uses that cohomology groups with discrete coefficients are discrete and that a continuous module is finite if and only if it is both discrete and profinite.
\end{proof}

We are now ready to state and prove the main result of this section:

\begin{theorem}\label{thm:profintosolid}
Let $G$ be a profinite group and $\ring$ a profinite ring.
\begin{enumerate}
\item If $\widehat{\otimes}_{\ring}$ denotes the tensor product of profinite $\ring\llbracket G \rrbracket$-modules and $\otimes_{\ring}^{\solid}$ the tensor product of solid $\rg$-modules, then
\[ \underline{M \widehat{\otimes}_{\ring} N} \cong \underline{M} \st \underline{N} \quad \text{and} \quad \underline{M \widehat{\otimes}_{\ring \llbracket G \rrbracket} N} \cong \underline{M} \stg \underline{N} \]
for any chain complexes of profinite $\ring\llbracket G \rrbracket$-modules $M$ and $N$. In particular, there is an isomorphism $H_{\bullet}^{\ring}(G, -) \cong \mathbf{H}_{\bullet}^{\cR}(G, -)$.
\item There is an isomorphism
\[ \mathrm{RHom}_{\ring \llbracket G \rrbracket}(M, N) \cong \mathrm{RHom}_{\rg}(\underline{M}, \underline{N}) \]
for any chain complex of profinite $\ring\llbracket G \rrbracket$-modules $M$ and any chain complex of discrete or profinite $\ring \llbracket G \rrbracket$-modules $N$. In particular, there is an isomorphism $H_{\ring}^{\bullet}(G, -) \cong H_{\cR}^{\bullet}(G, -)$.
\item If one endows the Hom-functors of continuous modules with the compact-open topology, then there is an isomorphism of chain complexes of condensed $\cR$-modules
\[ \underline{\mathrm{RHom}_{\ring \llbracket G \rrbracket}(M, N)} \cong \rhg(\underline{M}, \underline{N}) \]
for any chain complex of profinite $\ring\llbracket G \rrbracket$-modules $M$ and any chain complex of discrete or profinite $\ring \llbracket G \rrbracket$-modules $N$. If $G$ satisfies condition (FF), then this descends to an isomorphism of solid $\cR$-modules
\[ \underline{H_{\ring}^{\bullet}(G, N)} \cong \mathbf{H}_{\cR}^{\bullet}(G, \underline{N}) \]
for any profinite or discrete $\ring \llbracket G \rrbracket$-module $N$.
\end{enumerate}
\end{theorem}

Before proving this, we remark on the novelty of this theorem in some detail. The second assertion of the above theorem can be deduced from \cite[Theorem B]{bri25}. The first two chapters of Brink's article are the very first source to rigorously define the category-theoretic formalism underlying condensed mathematics while the last chapter investigates forms of condensed group cohomology over $\mathbb{Z}$. As stated, Brink's Theorem B only pertains to the coefficient ring $\mathbb{Z}$, but it can also be formulated over any profinite coefficient ring $\ring$. The third assertion of the above theorem is a lifting to derived functors of Tang's main result \cite{tang24}. The statements we need for derived functors are mostly implicit in his proofs, and the above theorem serves to extract them. Tang also indicates an approach to the first assertion.

\begin{proof}[Proof of Theorem~\ref{thm:profintosolid}]
One deduces the second assertion from Lemma~\ref{lem:profinitembeddings} and Lemma~\ref{lem:analyticgrouprings} where one uses additionally \cite[Proposition~3.20]{tang24} to prove the first assertion. The statement on derived Hom-functors in the third assertion follows from~\cite[Proposition~4.2]{Scholze2019} and~\cite[Proposition~3.12]{tang24}. Therefore, there is an isomorphism of condensed $\cR$-modules $\underline{H_{\ring}^{\bullet}(G, N)} \cong \mathbf{H}_{\cR}^{\bullet}(G, \underline{N})$. Here, $H_{\ring}^k(G, N)$ is profinite for profinite $N$ by Lemma~\ref{lem:profincohomgroups} and is discrete for discrete $N$. Therefore, the isomorphism of cohomology groups is one of solid $\cR$-modules by Lemma~\ref{lem:analyticstructure} and Lemma~\ref{lem:profinitembeddings}.
\end{proof}

\section{Finiteness conditions in condensed mathematics}\label{sec:finitenessconds}
The subject of this section is one of the first investigations of homological finiteness properties in condensed mathematics. It serves the purpose of proving two milestones towards our main result and constitutes a generalisation of the account~\cite[Section~3]{Anschuetz2020}. The fist milestone is that all profinite groups are of type $\mathrm{CP}_{\infty}$ over every profinite ring $\cR$, which is necessary for the proof of Theorem~\ref{thmA}. As an intermediate result, we prove that every solid $\rg$-module can be obtained as an direct limit of condensates of profinite $\ring \llbracket G \rrbracket$-modules. The second milestone is demonstrating that the cohomological dimension of $G$ for profinite modules is equal to that for solid modules, which is necessary in the proof of Corollary~\ref{corC}.

\subsection{Profinite groups are of type $\mathrm{CP}_\infty$}
We first prove that type $\mathrm{CP}_{\infty}$ is of a different nature than type $\mathrm{FP}_{\infty}$:

\begin{prop}\label{prop:solidandcompact}
For every infinitely generated profinite group $G$, the profinite $\widehat{\Z}\llbracket G \rrbracket$-module $\widehat{\Z}\llbracket G \rrbracket$ is finitely presented, but not a compact object. In other words, $\widehat{\mathbb{Z}} \llbracket G \rrbracket$ is of type $\mathrm{FP}_{\infty}$, but not of type $\mathrm{CP}_{\infty}$ in the category of profinite $\widehat{\mathbb{Z}} \llbracket G \rrbracket$-modules.
\end{prop}

\begin{proof}
Because $\widehat{\Z}\llbracket G \rrbracket$ is finitely presented, it remains to show that it is not compact. Let $G_0 \leq G_1 \leq \cdots$ be finitely generated subgroups whose union is dense in $G$. We obtain a direct system $\{\hat{\Z}\llbracket G/G_i \rrbracket\}_{i \geq 0}$ of profinite $\hat{\Z}\llbracket G \rrbracket$-modules whose colimit in the profinite category is $\hat{\Z}$ by~\cite[Lemma~4.4]{coo16b}. Therefore,
\[ \mathrm{Hom}_{\widehat{\Z}\llbracket G \rrbracket} \Big(\hat{\Z}\llbracket G \rrbracket, \varinjlim_{i \geq 0} \widehat{\Z}\llbracket G/G_i \rrbracket\Big) = \widehat{\Z} \, . \]
By the proof of \cite[Lemma~4.4]{coo16b}, the colimit of $\{\mathrm{Hom}_{\widehat{\Z}\llbracket G \rrbracket}(\hat{\Z}\llbracket G \rrbracket,\hat{\Z}\llbracket G/G_i \rrbracket)\}_{i \geq 0}$, that is, the colimit of $\{\hat{\Z}\llbracket G/G_i \rrbracket\}_{i \geq 0}$ as abstract modules, is not $\hat{\Z}$. In particular, the profinite $\hat{\Z} \llbracket G \rrbracket$-module $\hat{\Z} \llbracket G \rrbracket$ is not compact.
\end{proof}

The first milestone generalises~\cite[Lemma~3.1]{Anschuetz2020} and reads:

\begin{theorem}\label{thm:cpoo}
If $G$ is a profinite group and $\ring$ a commutative profinite ring, then $G$ is of type $\mathrm{CP}_{\infty}$ over $\cR$ for solid modules.
\end{theorem}

\begin{proof}
This theorem could be deduced from Lemma~3.5.11 and Lemma~3.5.16 in~\cite{bri25}, but we provide a more direct proof for the reader's convenience. According to~\cite[p.~206]{rib10} there is a projective resolution of the profinite $\ring\llbracket G \rrbracket$-module $\ring$
\[ \dots {} \rar \big(\ring\llbracket G \rrbracket\big)\llbracket G ^3\rrbracket \rar \big(\ring\llbracket G \rrbracket\big)\llbracket G ^2\rrbracket \rar \big(\ring\llbracket G \rrbracket\big)\llbracket G \rrbracket \rar \ring\llbracket G \rrbracket \rar \ring \]
called the inhomogeneous bar resolution. We see that $\underline{\big(\ring\llbracket G \rrbracket\big)\llbracket G ^j\rrbracket} = \rg[G^j]^{\solid}$ by the second assertion of Lemma~\ref{lem:profinitembeddings} and the second assertion of Lemma~\ref{lem:analyticgrouprings}. Because the condensation functor on profinite $\ring\llbracket G \rrbracket$-modules is exact by the first assertion of Lemma~\ref{lem:profinitembeddings}, it turns the above into a resolution of the solid $\rg$-module $\cR$
\[ \dots {} \rar \big(\rg \big)[G^3]^{\solid} \rar \big(\rg \big)[G^2]^{\solid} \rar \big(\rg \big)[G]^{\solid} \rar \rg \rar \cR. \]
By Lemma~\ref{lem:coeffringanalytic}, the terms in this resolution are compact projective and, thus, $G$ is of solid type $\mathrm{CP}_{\infty}$ over $\cR$.
\end{proof}

\subsection{Profinite equals solid cohomological dimension}
In this section we show that the profinite cohomological dimension of a profinite group agrees with the solid cohomological dimension of its condensate, $\cd_{\ring}(G) = \cd_{\cR}(\cG)$,

\subsubsection{Inverse limits and condensation}
Proving our result on cohomological dimensions requires some preliminary results of independent interest.
\begin{proposition}[{\cite[Proposition 3.2]{Anschuetz2020}}]\label{prop:invlimexact}
Inverse limits of condensates of compact abelian groups are exact where the same holds for compact modules over a $T1$ topological ring $L$. Equivalently, if $\lbrace A_i \rbrace_{i \in I}$ is any inverse system of compact abelian groups or of compact $L$-modules, then $R^j \varprojlim_{i \in I} \underline{A_i} = 0$ for $j > 0$.
\end{proposition}

\begin{proof}
The result has only been proved for compact abelian groups in~\cite{Anschuetz2020}. Assume that $\lbrace A_i \rbrace_{i \in I}$ is an inverse system of compact $L$-modules. According to~\cite[p.~15]{Scholze2019}, inverse limits are calculated pointwise in $\mathrm{Cond}(\mathcal{C})$ for any category $\mathcal{C}$ admitting all direct limits. Thus,
\[ \big(\varprojlim_{i \in I} \underline{A_i}\big)(S) = \varprojlim_{\mathbf{Mod}(\underline{L}(S))} \underline{A_i}(S) \]
for any profinite space $S$. The latter can be written as the abelian group $\varprojlim_{\mathbf{Ab}} \underline{A_i}(S)$ with an $\underline{L}(S)$-action. Since $R^j \varprojlim_{\mathbf{Ab}} \underline{A_i} = 0$ for $j > 0$, we conclude that $R^j \varprojlim_{\mathbf{Mod}(\underline{L})} \underline{A_i} = 0$.
\end{proof}

We prove that every solid module can be written as a direct limit of condensates of profinite modules, which generalises~\cite[Lemma~3.7]{Anschuetz2020}. We do so by using the notion of Ind-categories as defined in~\cite[Definition~6.1.1]{kas06}. In a nutshell, if $\mathcal{C}$ is a small category, then $\mathrm{Ind} \, \mathcal{C}$ consists of all ``formal'' direct limits of objects in $\mathcal{C}$. However, this cannot be readily extended to categories $\mathcal{C}$ that are not small, as can be inferred from the warning in the Introduction on page~1 of~\cite{kas06}. Because the category of profinite modules is not small, we only consider categories of profinite modules each less than a certain size and their associated Ind-categories.

\begin{lemma}\label{lem:indprofinites}
Let $\mathcal{F}_{\ring\llbracket G \rrbracket}$ denote the category of finite discrete $\ring\llbracket G \rrbracket$-modules, and let $\kappa$ be an uncountable strong limit cardinal. Let $\Pro(\finitemodules)$ be the full subcategory of $\SolidLG^{\heartsuit}$ consisting of condensates of profinite $\ring\llbracket G \rrbracket$-modules, and let $\Pro^\kappa (\finitemodules)$ be the full subcategory spanned by those profinite modules that can be written as a $\kappa$-small inverse limit of objects in $\finitemodules$. Then the condensation functor $\underline{-}\colon \mathcal{F}_{\ring\llbracket G \rrbracket} \rar \SolidLG^{\heartsuit}$ induces a fully faithful functor
\[ \mathrm{Ind} \, \mathrm{Pro}^\kappa(\underline{-})\colon \mathrm{Ind} \, \mathrm{Pro}^\kappa(\mathcal{F}_{\ring\llbracket G \rrbracket}) \rar \SolidLG^{\heartsuit} \]
Taking the colimit over all $\kappa$, this gives an equivalence $\colim_\kappa \Ind \Pro^\kappa (\finitemodules) \cong \SolidLG^{\heartsuit}$.
\end{lemma}

\begin{proof}
The the condensation functor extends to a functor
\[ \mathrm{Pro}(\underline{-})\colon \mathrm{Pro}(\mathcal{F}_{\ring\llbracket G \rrbracket}) \rar \SolidLG^{\heartsuit} \, . \]
Let us show that this functor renders $\mathrm{Pro}(\mathcal{F}_{\ring\llbracket G \rrbracket})$ a full subcategory of $\SolidLG^{\heartsuit}$. As $\mathrm{Pro}(\underline{-})$ is exact by Proposition~\ref{prop:invlimexact}, it remains to prove that it is fully faithful. Since inverse limits are constructed pointwise in $\SolidLG^{\heartsuit}$~\cite[p.~15]{Scholze2019}, every object in the essential image of $\mathrm{Pro}(\underline{-})$ is isomorphic to the condensate of a profinite $\ring\llbracket G \rrbracket$-module. By the first assertion of Lemma~\ref{lem:profinitembeddings}, the Hom-functors of profinite $\ring\llbracket G \rrbracket$-modules are isomorphic to those of their condensates. Therefore, for any inverse systems $\lbrace A_i \rbrace_{i \in I}$, $\lbrace B_j \rbrace_{j \in J}$ in $\mathcal{F}_{\ring\llbracket G \rrbracket}$, one concludes that
\[ \mathrm{Hom}_{\rg}(\varprojlim_{i \in I} \underline{A_i}, \varprojlim_{j \in J} \underline{B_j}) \cong \mathrm{Hom}_{\ring\llbracket G \rrbracket}(\varprojlim_{i \in I} A_i, \varprojlim_{j \in J} B_j) = \varprojlim_{j \in J} \mathrm{Hom}_{\ring\llbracket G \rrbracket}(\varprojlim_{i \in I} A_i, B_j) \, . \]
It follows from the universal property of direct limits that there is a homomorphism
\[  \varinjlim_{i \in I} \mathrm{Hom}_{\ring\llbracket G \rrbracket}(A_i, B_j) \rar \mathrm{Hom}_{\ring\llbracket G \rrbracket}(\varprojlim_{i \in I} A_i, B_j) \, . \]
Because any homomorphism $\varprojlim_{i \in I} A_i \rightarrow B_j$ factors through an $A_i$ by~\cite[Lemma~1.1.16]{rib10}, this is an isomorphism. Hence
\[ \mathrm{Hom}_{\rg}(\varprojlim_{i \in I} \underline{A_i}, \varprojlim_{j \in J} \underline{B_j}) \cong \varprojlim_{j \in J} \, \varinjlim_{i \in I} \mathrm{Hom}_{\ring\llbracket G \rrbracket}(A_i, B_j) \, , \]
$\mathrm{Pro}(\underline{-})$ is fully faithful and $\mathrm{Pro}(\mathcal{F}_{\ring\llbracket G \rrbracket})$ a full subcategory of $\SolidLG^{\heartsuit}$. 

Note that the condensation functor $\underline{-}$ on profinite $\ring\llbracket G \rrbracket$-modules is exact and fully faithful by the first assertion of Lemma~\ref{lem:profinitembeddings}. Thus, profinite $\ring\llbracket G \rrbracket$-modules form also a full subcategory of $\SolidLG^{\heartsuit}$. One can prove analogously to above that $\mathrm{Pro}(\mathcal{F}_{\ring\llbracket G \rrbracket})$ is equivalent to the category of profinite $\ring\llbracket G \rrbracket$-modules where one uses~\cite[pp.~5--6]{rib10} instead of Proposition~\ref{prop:invlimexact}. The condensation functor takes over this equivalence into the category $\SolidLG^{\heartsuit}$. 

We show that $\mathrm{Ind} \, \mathrm{Pro}^\kappa (\mathcal{F}_{\ring\llbracket G \rrbracket})$ is a full subcategory of $\SolidLG^{\heartsuit}$. By (\cite{HTT}, Proposition 5.3.5.11), we just need to show that $\Pro^\kappa (\finitemodules)$ consists of compact objects. By the above, we may consider any element of $\mathrm{Pro}(\mathcal{F}_{\ring\llbracket G \rrbracket})$ as condensate $\underline{M}$ of a profinite $\ring\llbracket G \rrbracket$-module. As a profinite module, it admits a surjective $\ring\llbracket G \rrbracket$-module homomorphism $f\colon \big(\ring\llbracket G \rrbracket\big)\llbracket M \rrbracket \rar M$ where the the condensate $\underline{\big(\ring\llbracket G \rrbracket\big)\llbracket M \rrbracket} = \big(\rg \big)[M]^{\solid}$ is compact by the second assertion of Lemma~\ref{lem:profinitembeddings}. Since there is an analogous surjective homomorphism for $\mathrm{Ker}(f)$ and the functor $\mathrm{Pro}(\underline{-})$ is exact, $\underline{M}$ can be written as a cokernel of an $\rg$-homomorphism $\big(\rg \big)[\mathrm{Ker}]^{\solid} \rar \big(\rg \big)[M]^{\solid}$. Hence, $M$ is a finite colimit of compact objects so it is compact.

We now show that these functors are jointly essentially surjective. Let $M$ be any object in $\SolidLG^{\heartsuit}$. Then there is a surjective $\rg$-homomorphism $\bigoplus_{k \in K} \big(\rg \big)[S_k]^{\solid} \rar M$ with $S_k$ extremally disconnected as $\big(\rg \big)[S_k]^{\solid}$ are generators of $\SolidLG^{\heartsuit}$ by definition. Given that $\big(\rg \big)[S_k]^{\solid} = \underline{\big(\ring\llbracket G \rrbracket\big)\llbracket S _k\rrbracket}$, we conclude that $\bigoplus_{k \in K} \big(\rg \big)[S_k]^{\solid}$ is contained in $\mathrm{Ind} \, \mathrm{Pro}(\mathcal{F}_{\ring\llbracket G \rrbracket})$. Since there is an analogous surjective homomorphism onto $\mathrm{Ker}(g)$, we see that $M$ is a cokernel of a morphism in $\mathrm{Ind} \, \mathrm{Pro}(\mathcal{F}_{\ring\llbracket G \rrbracket})$. Because $\mathrm{Ind} \, \mathrm{Pro}(\underline{-})$ is exact, we conclude that $M$ is an object in $\mathrm{Ind} \, \mathrm{Pro}(\mathcal{F}_{\ring\llbracket G \rrbracket})$ and that $\mathrm{Ind} \, \mathrm{Pro}(\underline{-})$ is an equivalence.
\end{proof}

\subsubsection{Cohomological dimension}
The following characterisation of projective dimension and thus of cohomological dimension was stated only for modules over a ring, but its proof can be easily extended to any abelian category.

\begin{lemma}[{\cite[Lemma~VIII.2.1]{bro82}}]\label{lem:cd}
For an object $M$ in an abelian category $\mathcal{C}$ the following conditions are equivalent.
\begin{enumerate}
    \item $\mathrm{proj} \, \mathrm{dim}_{\mathcal{C}} \, M \leq n$.
    \item $\mathrm{Ext}_{\mathcal{C}}^k(M, -) = 0$ for $k > n$.
    \item $\mathrm{Ext}_{\mathcal{C}}^{n+1}(M, -) = 0$.
    \item If $0 \rar K \rar P_{n-1} \rar {} \dots {} \rar P_0 \rar M \rar 0$ is any exact sequence in $\mathcal{C}$ with each $P_k$ projective, then $K$ is projective.
\end{enumerate}
\end{lemma}

The second milestone is a vast generalisation of~\cite[Proposition 3.9]{Anschuetz2020}.

\begin{theorem}\label{thm:cd}
Let $G$ be a profinite group. Denote by $\mathrm{cd}_{\ring}G$ the cohomological dimension of $G$ in the category of profinite $\ring\llbracket G \rrbracket$-modules and by $\mathrm{cd}_{\cR} G$ the cohomological dimension in $\SolidLG^{\heartsuit}$. Then $\mathrm{cd}_{\ring} G =\mathrm{cd}_{\cR} G$.
\end{theorem}

\begin{proof}
The proof strategy is based on using Lemma~\ref{lem:cd} to conclude that one cohomological dimension is bounded by the other. Namely, if there is a non-vanishing cohomology group in degree $k$, then the cohomological dimension is bounded below by $k$. Continuous cohomology $H_{\ring}^{\bullet}(G, M)$ is isomorphic to solid group cohomology $H_{\cR}^{\bullet}(G, \underline{M})$ for any profinite $\ring\llbracket G \rrbracket$-module $M$ by Theorem~\ref{thm:profintosolid}. Thus, by Lemma~\ref{lem:cd}, we conclude that $\cd_{\ring}(G) \leq \cd_{\cR}(G)$.

For the converse inequality, we know from Lemma~\ref{lem:cd} that there is a solid $\rg$-module $N$ such that $H_{\cR}^n(G, N) \neq 0$. According to Lemma~\ref{lem:indprofinites}, one can write $N = \varinjlim_{i \in I} \underline{N_i}$ with each $N_i$ a profinite $\ring\llbracket G \rrbracket$-module. Theorem~\ref{thm:profintosolid} and Theorem~\ref{thm:cpoo} imply that
\[ \varinjlim_{i \in I} \, H_{\ring}^n(G, N_i) \cong \varinjlim_{i \in I} \, H_{\cR}^n(G, \underline{N}_i) = H_{\cR}^n(G, N) \neq 0 \, . \]
For the direct limit to be non-zero one of the terms $H_{\ring}^n(G, N_i)$ needs to be nonzero. By Lemma~\ref{lem:cd}, $\cd_{\cR}(G) \leq \cd_{\ring}(G)$ and thus $\cd_{\ring}(G) = \cd_{\cR}(G)$. 
\end{proof}

\section{Derived functors, flatness and induced modules}\label{sec:varia}
This section proves the preliminary results that are necessary to demonstrate Theorem~\ref{thmA} and Theorem~\ref{thmB} in the next section. First, this involves a result on bounded chain complexes. Then we establish the identities on tensor products and internal Hom-functors of solid modules required for the Tate spectral sequence in Theorem~\ref{thm:tate}. We investigate when profinite and solid modules are flat, which is essential to Cohen--Macaulay groups in Definition~\ref{def:cm}, to Theorem~\ref{thmA} and to Theorem~\ref{thmB}. Lastly, we prove that the homology of induced solid modules vanishes, which is necessary in the proofs of Theorem~\ref{thmA} and Theorem~\ref{thmB}.

\subsection{Bounded chain complexes}
The following result on bounded chain complexes is used in the proof of the Tate spectral sequence to exploit condition (FF) together with $\cd_{\ring}(G) < \infty$.

\begin{lemma}\label{lem:finitecomplex}
If $R$ is a profinite ring, then any bounded chain complex of discrete $R$-modules with finite homology is quasi-isomorphic to a complex of finite $R$-modules.
\end{lemma}

\begin{proof}
Let $C^\bullet$ be a complex of discrete $R$-modules and $x \in C^i$. The cyclic $R$-submodule generated by $x$ is both discrete by assumption and compact (as $R$ is a compact ring), hence finite. We can therefore write $C^\bullet$ as the filtered colimit of subcomplexes of finite discrete modules $F_{\Vec{x}}^\bullet$ generated by finitely many elements $\Vec{x}$. Since taking homology commutes with filtered colimits, we have $H^i(C^\bullet) = \varinjlim_{\Vec{x}} H^i(F_{\Vec{x}}^\bullet)$. 
    
As $H^i(C^\bullet)$ is finite for all $i$ and $C^\bullet$ is bounded, the total homology $H^\bullet(C^\bullet)$ is generated by finitely many classes. Picking a finite set of cycles $\Vec{y}$ representing a set of generators for $H^\bullet(C^\bullet)$ we obtain a finite subcomplex $F_{\Vec{y}}^\bullet \subset C^\bullet$ such that the map induced on cohomology by the inclusion $F_{\Vec{y}}^\bullet \subset C^\bullet$ is surjective.

Let $K^\bullet = \ker(H^\bullet(F_{\Vec{y}}^\bullet) \to H^\bullet(C^\bullet))$. Because $F_{\Vec{y}}^\bullet$ is a finite complex, $K^\bullet$ is a finite graded module. An element in $K^i$ is represented by a cycle $z \in F_{\Vec{y}}^i$ that becomes a boundary in $C^i$, meaning $z=d(c)$ for some chain $c\in C^{i-1}$. Pick a finite set of chains $\Vec{c}$ in $C^\bullet$ whose boundaries represent a generating set for $K^\bullet$, and consider the subcomplex $F^\bullet_{\Vec{y}\cup \Vec{c}}$. This is still a finite complex, and the surjection $H^\bullet(F_{\Vec{y}\cup \Vec{c}}^\bullet) \to H^\bullet(C^\bullet)$ has a trivial kernel by construction. This means that $F_{\Vec{y}\cup \Vec{c}}^\bullet$ is a bounded complex of finite discrete $R$-modules that is quasi-isomorphic to $C^\bullet$.
\end{proof}

\subsection{The solid tensor product}
In this section we gather various results on the solid tensor product over $\ring$ and $\rg$ which are necessary for the proof of our core techinical result.

\subsubsection{Interaction between tensor products}
We will need to know how the tensor products $\st$ and $\stg$ interact.

\begin{lemma}\label{lem:tensorinteract}
	Let $G$ be a profinite group and $\ring$ be a commutative profinite ring. Then for objects $L,M,N$ in $\SolidLG$ we have
	\begin{align}
	    L \stg \left(M \st N\right) &\cong (L \st M) \stg N\label{eq:tensorinteract1}\\
                                    & \cong (L \stg M) \st N\label{eq:tensorinteract2}
	\end{align}
	in $\LSolid$. If $L$, $M$ and $N$ chain complexes of profinite $\ring \llbracket G \rrbracket$-modules, then
    \begin{align*}
	    L \widehat{\otimes}_{\ring \llbracket G \rrbracket} \left(M \widehat{\otimes}_{\ring} N\right) &\cong (L \widehat{\otimes}_{\ring} M) \widehat{\otimes}_{\ring \llbracket G \rrbracket} N \\
                                    & \cong (L \widehat{\otimes}_{\ring \llbracket G \rrbracket} M) \widehat{\otimes}_{\ring} N
	\end{align*}
    as chain complexes of profinite $\ring$-modules.
\end{lemma}

\begin{proof}
	Because the isomorphisms in the profinite case follow from those in the solid case by Theorem~\ref{thm:profintosolid}, we only prove Equation~\eqref{eq:tensorinteract1} and Equation~\eqref{eq:tensorinteract2}.
    For $A\in\Solid$, we have
    \begin{align*}
        \rh(L\stg (M\st N), A) &\cong \rhg(M\st N, \rh(L, A)) &\mbox{(Eq.~\eqref{eq:adjunction1})}\\
        &\cong \rhg(N,\rh(M,\rh(L,A)))&\mbox{(Eq.~\eqref{eq:adjunction3})}\\
        &\cong \rhg(N,\rh(L\st M, A))&\mbox{(Eq.~\eqref{eq:adjunction4})}\\
        &\cong \rh((L\st M)\stg N ,A)&\mbox{(Eq.~\eqref{eq:adjunction1})};
    \end{align*}
    and similarly
     \begin{align*}
     \rh(L\stg (M\st N), A) &\cong \rhg(L, \rh(M\st N, A)) &\mbox{(Eq.~\eqref{eq:adjunction1a})}\\
        &\cong \rhg(L,\rh(M,\rh(N,A)))&\mbox{(Eq.~\eqref{eq:adjunction4})}\\
        &\cong \rh(L\stg M,\rh(N, A))&\mbox{(Eq.~\eqref{eq:adjunction1a})}\\
        &\cong \rh((L\stg M)\st N ,A)&\mbox{(Eq.~\eqref{eq:adjunction4})}.
    \end{align*}
    By the Yoneda lemma, Equation~\eqref{eq:tensorinteract1} and Equation~\eqref{eq:tensorinteract2} follow.
\end{proof}

\subsubsection{Tensor products and inverse limits}
We also need the following lemma describing how the solid tensor product interacts with inverse limits.

\begin{lemma}\label{lem:tensorsandlimits}
Let $R$ be a profinite ring and let $M = \varprojlim_{i \in I} M_i$, $N = \varprojlim_{j \in J} N_j$ be inverse limits of profinite $R$-modules. Then
\[ \underline{M} \otimes_{\underline{R}}^{\solid} \underline{N} \cong \varprojlim_{(i, j) \in I \times J} \underline{M_i} \otimes_{\underline{R}}^{\solid} \underline{N_j}. \]
\end{lemma}
\begin{proof}
    The underived statement is \cite[Proposition~3.20]{tang24}. By Proposition~\ref{prop:invlimexact}, the higher derived functors of the inverse limit functor vanish, and the statement follows.
\end{proof}

\subsubsection{Flatness in the solid setting}
The Cohen--Macaulay property involves flatness. In general, it is tricky to determine which modules are flat, but more straightforward to determine which are projective. For examples of duality groups such as in Section~\ref{sec:exls}, it is easier to determine whether a proposed dualising module is projective than flat. Thus, the next results are dedicated to proving that projective solid modules are flat.

\begin{lemma}\label{lem:profinprojflat}
(\cite[Proposition~3.1]{bru66}) Let $G$ be a profinite group and let $\ring$ be a commutative profinite ring. Write $R$ for either $\ring$ or $\ring \llbracket G \rrbracket$. Then a profinite $R$-module is projective if and only if it is flat.
\end{lemma}

\begin{lemma}\label{lem:solidprojflat}
Let $G$ be a profinite group and let $\ring$ be a commutative profinite ring. Write $R$ for either $\ring$ or $\ring \llbracket G \rrbracket$. Then every projective solid $\underline{R}$-module is flat.
\end{lemma}

\begin{proof}
For every profinite space $S$ the free profinite $R$-module $R \llbracket S \rrbracket$ is flat in the profinite category by Lemma~\ref{lem:profinprojflat}. Thus, every free solid $\underline{R}$-module $R[S]^{\solid}$ is flat on condensates of profinite $R$-modules by Lemma~\ref{lem:profinitembeddings}. Note that direct limits in $\Solid_R^{\heartsuit}$ are exact due to~\cite[Theorem~1.10]{Scholze2019} and that morphisms in an Ind-category $\mathrm{Ind} \, \mathcal{C}$ can be written as direct limits of morphisms in $\mathcal{C}$ by~\cite[Proposition~6.1.13]{kas06}. It follows from this and the proof of Lemma~\ref{lem:indprofinites} that any short exact sequence in $\Solid_{\underline{R}}^{\heartsuit}$ can be written as a direct limit of short exact sequences of condensates of profinite $R$-modules. Since $\otimes_R^{\solid}$ is a left adjoint by Equation~\eqref{eq:adjunction1}, it preserves colimits and hence distributes over coproducts. Again, as direct limits are exact, we conclude that free solid $\underline{R}$-modules are flat in all of $\Solid_R^{\heartsuit}$. Given that the solid tensor product distributes over coproducts, and coproducts are exact by~\cite[Theorem~1.10]{Scholze2019}, coproducts of free solid $\underline{R}$-modules are flat. Because projective solid $\underline{R}$-modules are retracts of coproducts of free solid $\underline{R}$-modules by~\cite[Theorem~1.10]{Scholze2019}, they are flat.
\end{proof}

The following result is needed to prove Corollary~\ref{corC}. More specifically, we use it to demonstrate that a profinite group is profinite Cohen--Macaulay if and only if it is solid Cohen--Macaulay.

\begin{lemma}\label{lem:flatness}
Let $G$ be a profinite group, $\ring$ a commutative profinite ring. Then a profinite $\ring$-module $M$ is flat with respect to $\widehat{\otimes}_{\ring}$ if and only if the solid $\cR$-module $\underline{M}$ is flat with respect to $\st$. 
\end{lemma}

\begin{proof}
If $\underline{M}$ as a solid $\cR$-module is flat, then we conclude by Theorem~\ref{thm:profintosolid} and Lemma~\ref{lem:profinitembeddings} that $M$ as a profinite $\ring$-module is flat. It remains to prove the other implication where we assume that $M$ is a flat profinite $\ring$-module. The solid $\cR$-module $\underline{M}$ is flat on condensates of profinite $\ring$-modules by Lemma~\ref{lem:profinitembeddings}. We conclude that $\underline{M}$ is flat in all of $\SolidLG^{\heartsuit}$ as in the proof of Lemma~\ref{lem:solidprojflat} except that we utilise Equation~\eqref{eq:adjunction3} instead of Equation~\eqref{eq:adjunction1}.
\end{proof}

In order to prove Theorem~\ref{thmA} and Theorem~\ref{thmB}, we need that the dualising modules of duality groups are flat. This in turn requires a version of Shapiro's Lemma, which has the following preliminary:

\begin{lemma}\label{lem:tensorinducedmod}
    Let $M\in\SolidLG$. The solid $\rg$-module $\rg \st M$ is canonically isomorphic to the induced solid module $\rg \st M_0$, where $M_0$ denotes the underlying solid abelian group of $M$ equipped with the trivial $G$-action.
\end{lemma}

\begin{proof}
    We apply the Yoneda lemma. For any objects $N,A$ in $\SolidLG$ we have
    \[
    \rhg(\rg \st N, A) \cong\rhg(\rg, \rh(N,A))\cong \rh(N,A),
    \]
    and specialising to $N=M$ and $N=M_0$ gives 
    \begin{align*}
    \rhg(\rg\st M,A) &\cong \rh(M,A) \\
    &= \rh(M_0,A) \cong \rhg(\rg\st M_0,A),
    \end{align*}
    using that $\rh$ does not depend on the $G$-module structure.
\end{proof}

Our version of Shapiro's Lemma reads:

\begin{lemma}\label{lem:Shapiro}
    Let $M=\rg \st \bar{M}$ be a solid $\rg$-module induced from a solid abelian group $\bar{M}$. Then 
    \[
    \mathbf{H}_k^{\cR}(G, M)=\begin{cases}\bar{M} & \mbox{when $k$=0}\\0& \mbox{otherwise.}     
    \end{cases}
    \] 
\end{lemma}

\begin{proof}
    The homology of $G$ with coefficients in $M$ is computed from
    \[
    {\cR} \stg M ={\cR} \stg (\rg \st \bar{M}) \cong ({\cR} \stg \rg) \st \bar{M}= \bar{M},
    \]
    which is a chain complex concentrated in degree zero.
\end{proof}

\section{Main results}\label{sec:mains}

This section establishes necessary and sufficient conditions for profinite groups to have duality. In the first subsection we prove the core isomorphism \eqref{eq:key iso} of derived functors from which we can deduce our Tate spectral sequence. The latter is the main tool in proving the main theorems, namely Theorem~\ref{thmA} and Theorem~\ref{thmB}, which is the subject of the second subsection. Recall that the main theorems establish duality for profinite groups with coefficients in both solid and continuous modules. In the last subsection, we prove that this duality generalise Bieri--Eckmann duality and Serre--Verdier--Tate duality to all profinite coefficient rings.

\subsection{The core isomorphism}
Recall that $D_{\cR}^{\bullet}=\rhg(\cR,\rg)$ denotes the solid dualising complex, and $D_{\ring}^{\bullet} =\mathrm{RHom}_{\ring \llbracket G \rrbracket}(\ring, \ring\llbracket G\rrbracket)	$ the profinite dualising complex. Our core technical result is the following:

\begin{theorem}[Core isomorphism]\label{thm:tate}
	Let $G$ be a profinite group and $\ring$ a commutative profinite ring such that $\cd_{\cR}(G) < \infty$ and satisfies condition (FF). Then there is an isomorphism
	\begin{equation}\label{e:tatess1}
		{\cR} \stg \left(M \st D_{\cR}^{\bullet} \right) \xrightarrow{\cong} \rhg({\cR}, M)
	\end{equation}
	that is natural in $M\in \SolidLG$. If $\underline{M}$ is a condensate of a chain complex of profinite $\ring \llbracket G \rrbracket$-modules, then this yields a natural isomorphism in the derived category of profinite $\ring \llbracket G \rrbracket$-modules
	\begin{equation}\label{e:tatess2}
		\ring \widehat{\otimes}_{\ring \llbracket G \rrbracket} \left(M \widehat{\otimes}_{\ring} D_{\ring}^{\bullet} \right) \xrightarrow{\cong} \mathrm{RHom}_{\ring \llbracket G \rrbracket}(\ring, M)	
	\end{equation}
	that is natural in $M$.
\end{theorem}

The profinite statement follows from the solid statement by Theorem~\ref{thm:profintosolid}. Although the third assertion of Theorem~\ref{thm:profintosolid} only implies that $D_{\cR}^{\bullet} \cong \underline{D_{\ring}^{\bullet}}$ in $\LSolid$, it follows from its proof that this is an isomorphism in $\SolidLG$. Lemma~\ref{lem:tensorinteract} tells us that the left hand side of Equation~\eqref{e:tatess1} is isomorphic to $D_{\cR}^{\bullet} \otimes_{\rg}^{\solid} M$, leaving us to prove:

\begin{lemma}\label{lem:corelemma}
	Let $G$ be a profinite group and $\ring$ a commutative profinite ring such that $\cd_{\cR}(G) < \infty$ and satisfies condition (FF). Then there is an isomorphism
	\[
	D_{\cR}^{\bullet} \stg M \rar \rhg({\cR}, M)
	\]
	that is natural in $M$.
\end{lemma}

\begin{proof}
    We define for $M\in \SolidLG$ the natural transformation
    \begin{equation}\label{eq:defofphi}
        \Phi_M \colon D_{\cR}^{\bullet} \stg M \cong \rhg({\cR},\rg)\stg \rhg(\rg,M) \xrightarrow{\underline{\circ}} \rhg({\cR}, M),
    \end{equation}
    where $\underline{\circ}$ is the composition map (see Equation~\ref{eq:internalcomposition}).

    To prove that $\Phi_M$ is an isomorphism for all $M \in \SolidLG$, we first consider the case where $M$ is a compact projective generator of $\SolidLG^{\heartsuit}$. According to Lemma~\ref{lem:coeffringanalytic}, there is a set $I$ such that $M = \prod_I \rg$. Because the internal Hom-functor commutes with arbitrary products in the second variable, the target of our map evaluates to:
    \[ \rhg\left({\cR}, \prod_I \rg\right) \cong \prod_I \rhg({\cR}, \rg) = \prod_I D_{\cR}^{\bullet}. \]
    For the source, we first show that $D_{\cR}^{\bullet}$ can be written as an inverse limit of finite modules. As $\rhg({\cR},-)$ commutes with inverse limits, the complex $D_{\cR}^{\bullet}$ can be written as the limit over $D_{\ring_{\alpha}, U}^{\bullet} := \rhg({\cR}, \underline{\ring_\alpha [G/U]}^{\solid})$, where the $\Lambda_\alpha$ form an inverse system for $\ring$ and $U$ ranges over the open normal subgroups of $G$. By virtue of the finite cohomological dimension of $G$ and Lemma~\ref{lem:cd}, ${\cR}$ and the finite discrete module $\underline{\ring_\alpha [G/U]}$ are resolved by bounded chain complexes in $\SolidLG$, whence $D_{\ring_{\alpha}, U}^{\bullet}$ is bounded. By the hypothesis on the cohomology groups of $G$ and the third assertion of Theorem~\ref{thm:profintosolid}, the homology groups of $D_{\ring_{\alpha}, U}^{\bullet}$ are the finite discrete cohomology groups $\underline{H_{\ring}^k(G, \ring_{\alpha}[G/U])}$. Therefore, each $D_{\ring_{\alpha}, U}^{\bullet}$ is quasi-isomorphic to a bounded complex of finite modules according to Lemma~\ref{lem:finitecomplex}.
    
    Hence, replacing each $D_{\ring_{\alpha}, U}^{\bullet}$ by a complex of finite modules, we have
    \[ D_{\cR}^{\bullet} \stg \prod_I \rg \xrightarrow{\cong} \left( \varprojlim_{\alpha,U} D_{\ring_{\alpha}, U}^{\bullet} \right)\stg \prod_I \rg \, . \]
    Observe that all terms of the complexes $D_{\ring_{\alpha}, U}^{\bullet}$ and $\rg$ are condensates of profinite $\ring \llbracket G \rrbracket$-modules where products are inverse limits of finite products. Lemma~\ref{lem:tensorsandlimits} then implies
    \[
    \left( \varprojlim_{\alpha,U} D_{\ring_{\alpha}, U}^{\bullet} \right)\stg \prod_I \rg \xrightarrow{\cong} \varprojlim_{\alpha,U} \left(\prod_I  \left( D_{\ring_{\alpha}, U}^{\bullet} \stg \rg \right)\right) \xrightarrow{\cong} \varprojlim_{\alpha,U} \left(\prod_I D_{\ring_{\alpha}, U}^{\bullet}\right) \, .
    \]
    As limits commute with each other, this means that we have an isomorphism
    \[
    D_{\cR}^{\bullet} \stg \prod_I \rg \xrightarrow{\cong} \prod_I D_{\cR}^{\bullet}.
    \]
    By naturality and the universal property of products, the map $\Phi_{\prod_I \rg}$ factors as the product of the individual maps $\Phi_{\rg}$. The map $\Phi_{\rg} \colon D_{\cR}^{\bullet} \stg \rg \rar D_{\cR}^{\bullet}$ is just the right unitor of the monoidal structure, so an isomorphism. This means that $\Phi_{\prod_I \rg}$ is an isomorphism for any set $I$.

    This argument extends to all compact projective objects. Any compact projective object $P$ is a retract of a direct sum of products of the form $P_j = \prod_{I_j} \rg$. To see that $\Phi$ is an isomorphism on such a coproduct $\bigoplus_j P_j$, we must check that the functors on both sides commute with direct sums. On the one hand, the tensor product $D_{\cR}^{\bullet} \stg (-)$ is a left adjoint by Equation~\eqref{eq:adjunction2} and hence commutes with colimits. On the other hand, as $G$ is of type $\mathrm{CP}_{\infty}$ over $\cR$ for solid modules according to Theorem~\ref{thm:cpoo}, the functor $\rhg({\cR}, -)$ commutes with arbitrary direct sums. Because $\Phi_{P_j}$ is an isomorphism on each individual product component, $\Phi_{\bigoplus_j P_j}$ is also an isomorphism on their direct sum. Finally, because both the tensor product and the internal hom preserve retracts, the natural transformation $\Phi_P$ restricts to an isomorphism for all compact projective objects $P$.

    We prove inductively that for any bounded chain complex $M$ of compact projectives the map $\Phi_M$ is an isomorphism. Since the base case is bounded complexes of length $1$, assume that the claim holds for complexes of length less than $k$ and that $M$ is of length $k$. Consider the short exact sequence
    \[
    0 \to M_k \to M \to M/M_k \to 0 \, ,
    \]
    where $M_k$ denotes the $k^{\text{th}}$ term of $M$. Applying the five-lemma to the map $\Phi$ between the long exact sequences for $D_{\cR}^{\bullet} \stg -$ and $\rhg({\cR}, -)$ yields that $\Phi_M$ is an isomorphism. This extends to an arbitrary, potentially unbounded complex $M$ of compact projectives. By hypothesis, the group $G$ has finite cohomological dimension. This ensures that both $D_{\cR}^{\bullet} \stg (-)$ and $\rhg({\cR}, -)$ have finitely many non-vanishing homology groups. Consequently, both $D_{\cR}^{\bullet} \stg M$ and $\rhg({\cR}, M)$ depend only on a bounded truncation $\tau_{\geq a} \tau_{\leq b} M$.
    
    This extends to any chain complex $M$. The functors $D_{\cR}^{\bullet} \stg (-)$ and $\rhg({\cR}, -)$ commute with direct limits because the former is a left adjoint by Equation~\eqref{eq:adjunction2} and for the latter the group $G$ is of type $\mathrm{CP}_{\infty}$ over $\cR$. As both functors also commute with finite products, they commute with sifted colimits. Given that every object $M$ in $\SolidLG$ is obtained as sifted colimits of chain complexes of compact projective objects according to Lemma~\ref{lem:siftedcolimits}, $\Phi_M$ is an isomorphism.
\end{proof}

We do not expect that Theorem~\ref{thm:tate} can be proved for profinite modules only using methods from the profinite setting. One can adapt the proof of Lemma~\ref{lem:corelemma} from the condensed to the profinite setting until the point where the lemma's isomorphism holds for chain complexes $M$ of projective profinite modules. Any chain complex of arbitrary profinite modules can be written as direct limit of the former chain complexes by taking term-wise projective resolutions consisting of free profinite modules. In order to extend the isomorphism from Lemma~\ref{lem:corelemma} to these arbitrary chain complexes, one would require that every profinite group is of type $\mathrm{CP}_{\infty}$ within the profinite setting. However, this is doubtful because not every discrete group is of type $\mathrm{CP}_{\infty}$ or equivalently, of type $\mathrm{FP}_{\infty}$. This is a reason why our form of duality has not been achieved by any existing method from the literature thus far and why condensed mathematics has potential for group cohomology.

\subsection{Proof of the main theorems}

This subsection contains a proof of Theorem~\ref{thmA}, and provides two proofs of Theorem~\ref{thmB} that are of independent interest. On the one hand, Theorem~\ref{thmB} can be proved as stand-alone consequence of the profinite isomorphism in Theorem~\ref{thm:tate} using the same arguments to prove the Theorem~\ref{thmA} from the solid isomorphism. Therefore, we first demonstrate these main theorems in parallel. However, we remind the reader that the Tate spectral sequence was established only via the solid setting. Theorem~\ref{thmB} can also be deduced as a corollary of Theorem~\ref{thmA}, which highlights again the power of the condensed machinery and Theorem~\ref{thm:profintosolid}.

\begin{proof}[Proof of Theorem~\ref{thmA} and Theorem~\ref{thmB}]
    We prove both theorems in parallel. Theorem~\ref{thmA} lives entirely in the solid setting, Theorem~\ref{thmB} entirely in the profinite setting, but as the required arguments are identical, only the notation differs in either case. We concentrate on the solid case and provide references for the profinite one whenever these cases differ. In order to retrieve the profinite case from the below arguments, one uses Theorem~\ref{thm:profintosolid} to replace `$\st$' by `$\widehat{\otimes}_{\ring}$', `$\stg$' by `$\widehat{\otimes}_{\ring \llbracket G \rrbracket}$', `$\mathbf{H}_{\cR}^{\bullet}(G, -)$' by `$H_{\ring}^{\bullet}(G, -)$' and `$D_{\cR}$' by `$D_{\ring}$'. Following Lemma~\ref{lem:profinitembeddings} and its subsequent paragraph, one replaces any instance of `$\rg$' by `$\ring \llbracket G \rrbracket$'. At the very end, we prove that the duality isomorphism of Equation~\eqref{eq:profinduality} in Theorem~\ref{thmB} holds for discrete $\ring \llbracket G \rrbracket$-modules under the specified conditions.
    
    So far, we have not distinguished between underived and derived tensor products. As this becomes necessary now, we introduce for $M, N \in \SolidLG^{\heartsuit}$ the notation
    \[
    M \stheart N := H_0(M\st N).
    \]
    for the underived tensor product.
    
    To see that any Cohen--Macaulay group is a duality group, we consider the Tate spectral sequence associated to the isomorphism from Theorem~\ref{thm:tate}. If $P_{\bullet}$ denotes a projective resolution of the solid $\rg$-module ${\cR}$, then the left-hand side of this isomorphism can be written as
    \[
    P_\bullet \stg (M \st D_{\cR}^{\bullet}) \cong (P_\bullet \stg M) \st D_{\cR}^{\bullet} \, .
    \]
    according to Lemma~\ref{lem:tensorinteract}. First taking the differential along $D_{\cR}^{\bullet}$ on the double complex gives
    \[
    E_{\bullet\star}^1 = \big(P_\bullet\stg M \big) \st \mathbf{H}_{\cR}^{-\star}(G, \rg) \cong P_\bullet\stg \big(M \st \mathbf{H}_{\cR}^{-\star}(G, \rg)\big) \, ,
    \]
    where we are applying Lemma~\ref{lem:tensorinteract} again. Because the solid $\rg$-module $\mathbf{H}_{\cR}^{-\star}(G, \rg)$ is flat with respect to $\st$, the tensor product in $M \st \mathbf{H}_{\cR}^{-\star}(G, \rg)$ is underived. As $\mathbf{H}_{\cR}^{-\star}(G, \rg)$ is assumed to be concentrated in degree $n$, the Tate spectral sequence collapses on the next page so that
    \[
    \mathbf{H}_{\bullet}^{\cR}\big(G, M \stheart \mathbf{H}_{\cR}^n(G, \rg)\big) \cong \mathbf{H}_{\cR}^{n-\bullet}(G, M) \, .
    \]
    To see that any duality group is Cohen--Macaulay, we evaluate the duality isomorphism at the group ring $M = \rg$ and obtain that
    \[
    \mathbf{H}_{\cR}^{n-k}(G, \rg) \cong \mathbf{H}_k^{\cR}(G, \rg \stheart D_{\cR}) \, . 
    \]
    Observe that $\rg \stheart D_{\cR}$ is isomorphic to an induced module by Lemma~\ref{lem:tensorinducedmod} in the solid case and~\cite[Remark~3.3.3]{sym00} in the profinite case. Hence,
    \begin{equation*}
        \mathbf{H}_{\cR}^{k}(G, \rg) \cong \begin{cases} D_{\cR} &\mbox{ when }$k=n$\\ 0 & \mbox{otherwise}\end{cases}
    \end{equation*}
    by Lemma~\ref{lem:Shapiro} in the solid case and by~\cite[Lemma~3.3.4]{sym00} in the profinite case. This shows that $\mathbf{H}_{\cR}^{\bullet}(G, \rg)$ is indeed concentrated in degree $n$. 
    
    It remains to show that the solid $\rg$-module $D_{\cR} \cong \mathbf{H}_{\cR}^n(G, \rg)$ is flat with respect to $\st$. Let $0 \to A \to B \to C \to 0$ be a short exact sequence of solid $\rg$-modules. We aim to show that
    \begin{equation}\label{eq:sesforsure}
    0 \to A \stheart D_{\cR} \to B \stheart D_{\cR} \to C \stheart D_{\cR} \to 0    
    \end{equation}
    is exact. First we use Lemma~\ref{lem:Shapiro} in the solid case, \cite[Lemma~3.3.4]{sym00} in the profinite case and then duality to conclude that
    \[
    A \stheart D_{\cR} \cong H_0^{\cR}\big(G, \rg \stheart A \stheart D_{\cR} \big) \cong \mathbf{H}_{\cR}^n(G, \rg \stheart A ),
    \]
    with similar isomorphisms for $B$ and $C$. The cohomology groups on the right hand side are part of the long exact sequence associated to the short exact sequence
    \[
    0 \to \rg \stheart A \to \rg \stheart B \to \rg \stheart C \to 0 \, .
    \]
    That is, the sequence in Equation~\eqref{eq:sesforsure} is isomorphic to the middle portion of
    \begin{align*}
        \mathbf{H}_{\cR}^{n-1}(G, \rg \stheart C) \to &\mathbf{H}_{\cR}^n(G, \rg \stheart A) \to \mathbf{H}_{\cR}^n(G, \rg \stheart B) \to \\
        &\mathbf{H}_{\cR}^n(G, \rg \stheart C) \to \mathbf{H}_{\cR}^{n+1}(G, \rg \stheart A) \, . 
    \end{align*}
    Note that $\mathbf{H}_{\cR}^{n+1}(G, \rg \stheart A)$ vanishes as $\cd_{\cR}(G) = n$ (respectively, $\cd_{\ring}(G) = n$). Moreover,
    \[ \mathbf{H}_{\cR}^{n-1}(G, \rg \stheart C)\cong H_{1}^{\cR}\big(G, \rg \st C \stheart D_{\cR} \big)=0 \]
    by duality, Lemma~\ref{lem:Shapiro} in the solid case and~\cite[Lemma~3.3.4]{sym00} in the profinite case. Hence, the sequence in Equation~\eqref{eq:sesforsure} is exact, $D_{\cR}$ is flat with respect to $\st$ and $G$ is Cohen--Macaulay.

    For the last assertion of Theorem~\ref{thmB}, assume that $M$ is a discrete $\ring \llbracket G \rrbracket$-module, $G$ is of type $\mathrm{FP}_{\infty}$ over $\ring$ and $D_{\ring}$ is a finitely generated profinite $\ring \llbracket G \rrbracket$-module. These conditions imply that the tensor product $M \widehat{\otimes}_{\ring} D_{\ring} = M \otimes_{\ring} D_{\ring}$ and the homology groups $H_{\bullet}^{\ring}(G, M \widehat{\otimes}_{\ring} D_{\ring})$ are well defined by~\cite[pp.~378--379]{sym00}. This source also implies that the functors $H_{\bullet}^{\ring}(G, - \widehat{\otimes}_{\ring} D_{\ring})$ and $H_{\ring}^{\bullet}(G, -)$ commute with direct limits. Because every discrete $\ring \llbracket G \rrbracket$-module is a direct limit of finite $\ring \llbracket G \rrbracket$-modules by~\cite[Lemma~5.1.1]{rib10}, the duality isomorphism in Equation~\eqref{eq:profinduality} extends to all discrete $\ring \llbracket G \rrbracket$-modules $M$.
\end{proof}

\begin{proof}[Second proof of Theorem~\ref{thmB}]
    We use Theorem~\ref{thmA} in order prove Theorem~\ref{thmB}. The paragraph preceding Corollary~\ref{corC} proves that a profinite group $G$ is profinite Cohen--Macaulay if and only it is solid Cohen--Macaulay. According to Theorem~\ref{thmA}, the latter is equivalent to $G$ being a solid duality group. Theorem~\ref{thm:profintosolid} implies that $G$ is a profinite duality group. The last paragraph of the previous proof demonstrates that duality holds for any discrete $\ring \llbracket G \rrbracket$-module $M$ under the specified conditions.

    It remains to prove that any profinite duality group $G$ is also a solid duality group. By Theorem~\ref{thm:profintosolid}, the duality isomorphism
    \begin{equation}\label{eq:dualityforcondensates}
    \mathbf{H}_k^{\cR}(G, \underline{M} \st D_{\cR}) \rightarrow \mathbf{H}_{\cR}^{n-k}(G, \underline{M})
    \end{equation}
    holds for any condensate of a profinite $\ring \llbracket G \rrbracket$-module $M$. Any solid $\rg$-module can be written as a direct limit of condensates of profinite $\ring \llbracket G \rrbracket$-modules $N = \varinjlim_{i \in I} \underline{M_i}$ by Lemma~\ref{lem:indprofinites}. Recall that $G$ is type $\mathrm{CP}_{\infty}$ in the solid setting by Theorem~\ref{thm:cpoo} and the tensor products `$\st$' and `$\stg$' as left-adjoints commute with direct limits. It follows from~\cite[Theorem~1.10]{Scholze2019} and Lemma~\ref{lem:analyticrings} that direct limits of solid modules are exact. In particular, Equation~\eqref{eq:dualityforcondensates} extends to every solid $\rg$-module $M$, whence $G$ is a solid duality group.
\end{proof}

\subsection{Duality via cap- and cup-product pairings}

In the literature, duality has not only been established via the Cohen--Macaulay condition, but also via cap- and cup-product pairings. More specifically, Bieri--Eckmann duality in~\cite{bie73} has been formulated over a cap-product pairing while Serre--Verdier--Tate duality in~\cite{Serre1997}, \cite{Verdier1965} has been formulated over a cup-product pairing. In the profinite setting, these pairings have only been constructed over the coefficient ring $\Z_p$ \cite{Pletch1980}, \cite{ser94}, \cite{sym00}, \cite{Verdier1965}, \cite{Wilkes2019}. In this subsection, we generalise a cap- and a cup-product pairing from~\cite{Pletch1980} to every profinite coefficient ring $\ring$. In doing so, we also prove that the duality from Theorem~\ref{thmB} generalises Bieri--Eckmann duality and Serre--Verdier--Tate duality to the profinite ring $\ring$.

\subsubsection{Solid cap product}
First, we tackle the cap-product pairing which necessitates the construction of a cap product in the derived category $\SolidLG$. This derived solid cap product can be seen as a reformulation of the following fact for discrete groups: the edge homomorphism of the Tate spectral sequence coming from Theorem~\ref{thm:tate} is the inverse of the cap product with a fundamental class~\cite[Section VIII.10]{bro82}.

\begin{definition}[Solid cap product]
Let $G$ be a profinite group, $\ring$ a commutative profinite ring and $M, N$ objects in $\SolidLG$. Denote the canonical evaluation morphism by
\[ \mathrm{eval} \colon {\cR} \st \rhg({\cR}, N) \to N \, . \]
The cap product is defined, using Lemma~\ref{lem:tensorinteract} and the fact that $M\stg N \cong N\stg M$ (see \textsection\ref{sssec:tensorhoms}), as the map
\begin{align*}
    -\frown - \colon ({\cR} \stg M) \st &\rhg({\cR}, N) \\
    &\xrightarrow{\cong} M \stg \left({\cR} \st \rhg({\cR}, N)\right) \\
    &\xrightarrow{\id \st \mathrm{eval}}  M\stg N \\
    &\cong {\cR} \stg (M \st N).
\end{align*}
\end{definition}

We also have a notion of fundamental class:

\begin{definition}
The representative of the fundamental class is a preimage $e$ of the identity map under the specialisation of the quasi-isomorphism from Theorem~\ref{thm:tate}
\[ {\cR} \stg D_{\cR}^{\bullet} \xrightarrow{\simeq} \rhg({\cR}, {\cR}) \] 
to the trivial module ${\cR}$.
\end{definition}

Capping with the fundamental class is inverse to the isomorphism from Theorem~\ref{thm:tate}:

\begin{proposition}\label{prop:cap_product_quasi_iso}
    Let $G$ be a profinite group and $\ring$ a commutative profinite ring such that $\cd_{\cR}(G) < \infty$ and $G$ satisfies condition (FF). For any object $M$ in $\SolidLG$, the cap product with a representative $e$ of the fundamental class cycle is an isomorphism
\[ e \frown - \colon \rhg({\cR}, M) \xrightarrow{\cong} {\cR} \stg (D_{\cR}^{\bullet} \st M) \]
inverse to the isomorphism from Theorem~\ref{thm:tate}. 
\end{proposition}

\begin{proof}
Denote the natural isomorphism from Theorem~\ref{thm:tate} by
\[ \tilde{\Phi}_M \colon {\cR} \stg (D_{\cR}^{\bullet} \st M) \xrightarrow{\cong} \rhg({\cR}, M) \, . \]
The cap product $e \frown -$ determines a natural transformation
\[ C \colon \rhg({\cR}, -) \to {\cR} \stg (D_{\cR}^{\bullet} \st -) \, . \] 
Note that $\tilde{\Phi} \circ C$ is a natural endomorphism of $\rhg({\cR}, -)$. By the Yoneda lemma, any natural endomorphism of this functor is determined by its value on the identity morphism $\id_{{\cR}} \in \rhg({\cR}, {\cR})$. 

So let us evaluate $\tilde{\Phi}_{{\cR}} \circ C_{{\cR}}(\id_{{\cR}})$. When the element in the internal Hom-functor is the identity morphism $\id_{{\cR}}$, the evaluation
$$
{\cR} \st \rhg({\cR}, {\cR}) \xrightarrow{\mathrm{eval}_{\cR}} {\cR}
$$
reduces to the identity map on ${\cR}$. This collapses the evaluation step in the cap product $e\frown \id_{\cR}$, leaving us with $e \in {\cR} \stg D_{\cR}^{\bullet}$. By definition $\tilde{\Phi}_{{\cR}}(e) = \id_{{\cR}}$, so $\tilde{\Phi} \circ C$ is the identity on $\rhg({\cR}, -)$. Because $\tilde{\Phi}$ is an isomorphism by Theorem~\ref{thm:tate}, this implies that $C$ and $\tilde{\Phi}$ are mutually inverse.
\end{proof}

\subsubsection{Cap-product pairing}

For Bieri--Eckmann duality, the relevant profinite cap-product pairing has been defined only over $\Z_p$, but one can easily extend its definition to any profinite coefficient ring $\ring$:

\begin{definition}[Profinite cap product, $\text{\cite[pp.~62--63]{Pletch1980}}$]\label{defn:capprod}
Let $G$ be a profinite group and $\ring$ a commutative profinite ring such that $G$ satisfies condition (FF). Then there are natural continuous $\ring$-bilinear maps
\[ - \frown - \colon H_r^{\ring}(G, A) \times H_{\ring}^s(G, B) \rightarrow H_{r-s}^{\ring}(G, A \widehat{\otimes}_{\ring} B) \]
for any profinite $\ring \llbracket G \rrbracket$-module $A$, $B$ and every $r \geq s$ that are defined on the cochain level by
\begin{align*}
    \big( A \widehat{\otimes}_{\ring \llbracket G \rrbracket} \ring \llbracket G^{r+1} \rrbracket \big) \times C^s(G, B) &\rightarrow (A \widehat{\otimes}_{\ring} B) \widehat{\otimes}_{\ring \llbracket G \rrbracket} \ring \llbracket G^{r-s+1} \rrbracket \\
    \big(a \otimes (g_0, \dots, g_r), \sigma \big) &\mapsto \big(a \otimes \sigma(g_0, \dots, g_s)\big) \otimes (g_s, \dots, g_r) \, .
\end{align*}
\end{definition}

The duality from Theorem~\ref{thmB} can be expressed via the above cap-product pairing and thus generalises Bieri--Eckmann duality to every profinite coefficient ring $\ring$:

\begin{proposition}[Duality via a cap product]\label{prop:capduality}
Let $G$ be a profinite group and $\ring$ a commutative profinite ring such that $\cd_{\ring}(G) = n$ and $G$ satisfies condition (FF). Then $G$ is a profinite duality group if and only if there exists a profinite $\ring \llbracket G \rrbracket$-module $C$ and an element $e \in H_n^{\ring}(G, C)$ such that for every profinite $\ring \llbracket G \rrbracket$-module $M$ the pairing
\begin{equation}\label{eq:profcapwithe}
     e \frown - \colon H_{\ring}^k(G, M) \rightarrow H_{n-k}^{\ring}(G, C \widehat{\otimes}_{\ring} M)
\end{equation}
is an isomorphism for every profinite $\ring \llbracket G \rrbracket$-module $M$.    
\end{proposition}

\begin{proof}
Since the cap-product pairing already gives a duality isomorphism, it remains to prove that every profinite duality group possesses such a cap-product pairing. The duality isomorphism in Theorem~\ref{thmB} can be taken as the homology of the profinite isomorphism in Equation~\eqref{e:tatess2} from Theorem~\ref{thm:tate}. The subject of Proposition~\ref{prop:cap_product_quasi_iso} is that the inverse of the latter isomorphism is a derived cap-product pairing. Using Theorem~\ref{thm:profintosolid}, one can thus express the inverse of the profinite isomorphism in Theorem~\ref{thm:tate} as a derived cap-product pairing
\[ e' \frown - \colon \mathrm{RHom}_{\ring \llbracket G \rrbracket}(\ring, M) \xrightarrow{\cong} \ring \widehat{\otimes}_{\ring \llbracket G \rrbracket} (D_{\ring}^{\bullet} \widehat{\otimes}_{\ring} M) \]
in the profinite setting. Taking homology of this derived isomorphism gives the desired cap-product pairing where $C = D_{\ring}$ and the element $e$ is the class of $e'$.
\end{proof}

To characterise duality via a cup-product pairing shortly, we require that it can be already characterised by the cap-product pairing taking only finite coefficients:

\begin{corollary}\label{cor:finiteduality}
 Let $G$ be a profinite group and $\ring$ a commutative profinite ring such that $\cd_{\ring}(G) = n$ and $G$ satisfies condition (FF). Then  $G$ is a profinite duality group if and only if the cap-product pairing
\[ e \frown - \colon H_{\ring}^k(G, M) \rightarrow H_{n-k}^{\ring}(G, C \widehat{\otimes}_{\ring} M) \]
is an isomorphism for every {finite} $\ring \llbracket G \rrbracket$-module $M$ and every integer $k$.
\end{corollary}

\begin{proof}
By Proposition~\ref{prop:capduality}, it suffices to show that the cap product with the fundamental class $e$ from \eqref{eq:profcapwithe} is an isomorphism for all profinite modules if and only if it is an isomorphism for all finite modules. This is achieved by proving that both sides of the isomorphism commute with inverse limits. The cohomology groups $H_{\ring}^{n-k}(G, -)$ do so because inverse limits are exact by~\cite[Proposition~2.2.4]{rib10}. Moving on to the relevant homology groups, completed tensor products of profinite modules commute with inverse limits by~\cite[Lemma~5.5.1]{rib10}. As $G$ satisfies condition (FF), $D_\ring$ is profinite. Thus, $H_k^{\ring}(G, - \widehat{\otimes}_{\ring} D_{\ring})$ also commutes with inverse limits.
\end{proof}

\subsubsection{Cup-product pairing} 

We conclude this subsection by proving that duality can be expressed via a cup-product pairing and that it thus generalises Serre--Verdier--Tate duality to every profinite coefficient ring $\ring$. The cup product provides a way of combining cohomology classes from different modules:

\begin{definition}[Profinite cup product, $\text{\cite[pp.~282--284]{rib10}}$]
Let $G$ be a profinite group and $\ring$ a commutative profinite ring. For any discrete $\ring \llbracket G \rrbracket$-modules $M$, $N$ and every $r, s \geq 0$ there are $\ring$-bilinear maps of discrete $\ring$-modules
\[ H_{\ring}^r(G, M) \otimes_\ring H_{\ring}^s(G, N) \rightarrow H_{\ring}^{r+s}(G, M \otimes_{\ring} N) \]
called \emph{cup products}. On the cochain level, for $\phi \in C^r(G, M)$ and $\psi \in C^s(G, N)$, it is defined by by
\[ (\phi \smile \psi)(g_0, \dots, g_{r+s}) = \phi(g_0, \dots, g_r) \otimes\psi(g_r, \dots, g_{r+s}) \]
for $g_0, \dots, g_{r+s} \in G$.
\end{definition}

The cup product duality pairing is formulated using Pontryagin duality, which is a duality between profinite modules and discrete (torsion) modules. More specifically, endow $\mathbb{R}/\mathbb{Z}$ with the quotient topology, $\mathbb{Q}/\mathbb{Z}$ with the corresponding subspace topology and $M^{\ast} := \mathrm{Hom}_{\mathrm{cts}}(M, \mathbb{Q}/\mathbb{Z})$ with the compact-open topology. For every profinite ring $R$, the functor $(-)^{\ast} := \mathrm{Hom}_{\mathrm{cts}}(-, \mathbb{Q}/\mathbb{Z})$ forms a duality between profinite $R$-modules and and discrete $R$-modules, which comes with a continuous $R$-isomorphism $M \rightarrow M^{\ast \ast}$. Classically, this is often used to avoid considering derived functors of completed tensor products. It allows us to work entirely with cohomology of discrete modules, by thinking morally of $H^\bullet(G,M^\ast)^\ast$ as a replacement for the homology of $M$. Applying this logic to the module $H_n(G,D_\ring)$ containing the fundamental class, we see that we should replace the fundamental class by a map
$$
e'\colon H^{n}(G, D_\ring^*)\to \mathbb{Q}/\Z
$$
To use such a fundamental class for a duality pairing for an arbitrary module $A$, we therefore want to take (following Pletch who gave this definition over $\Z_p$) our cup product pairing with the fundamental class as follows:

\begin{definition}[Cup-product pairing, {\cite[p.~64]{Pletch1980}}]\label{defn:cupprod}
Let $G$ be a profinite group and $\ring$ a commutative profinite ring. For a finite $\ring \llbracket G \rrbracket$-module $A$ and a profinite $\ring \llbracket G \rrbracket$-module $C$, the \emph{cup-product pairing} with an element $e' \in H_{\ring}^{r+s}(G, C^{\ast})^{\ast}$ is the map$$\Psi_A(e') \colon H_{\ring}^s(G, \mathrm{Hom}_{\ring}(A, C^{\ast})) \rightarrow H_{\ring}^r(G, A)^{\ast}$$defined as the tensor-hom adjoint of the composition
\begin{align*}
   H_{\ring}^r(G, A) \otimes H_{\ring}^s(G, \mathrm{Hom}_{\ring}(A, C^{\ast})) &\xrightarrow{\smallsmile} H_{\ring}^{r+s}(G, A \otimes_{\ring} \mathrm{Hom}_{\ring}(A, C^{\ast}))\\ & \xrightarrow{\mathrm{eval}_{\ast}} H_{\ring}^{r+s}(G, C^{\ast}) \xrightarrow{e'} \mathbb{Q}/\mathbb{Z} \, , 
\end{align*}
where $\mathrm{eval}_{\ast}$ is induced by the evaluation homomorphism $A \otimes_{\ring} \mathrm{Hom}_{\ring}(A, C^{\ast}) \to C^{\ast}$.
\end{definition}

We now have everything we need to establish that duality can also be characterised through the cup product. We note that the proof of this proposition ultimately depends on Theorem~\ref{thm:tate}, for which we are only aware of a proof using condensed methods.

\begin{proposition}[Duality via a cup product]\label{prop:cupduality}
Let $G$ be a profinite group and $\ring$ a commutative profinite ring such that $\cd_{\ring}(G) = n$ and $G$ satisfies condition (FF). Then $G$ is a profinite duality group if and only if there exists a profinite $\ring \llbracket G \rrbracket$-module $C$ and an element $e \in H_{\ring}^n(G, C^{\ast})^{\ast}$ such that for every finite $\ring \llbracket G \rrbracket$-module $M$ the pairing
\[ \Psi(e) \colon H_{\ring}^k(G, \mathrm{Hom}_{\ring}(M, C^{\ast})) \rightarrow H_{\ring}^{n-k}(G, M)^{\ast} \]
is an isomorphism.
\end{proposition}

\begin{proof}
By Corollary~\ref{cor:finiteduality}, it suffices to prove the proposition using the classical equivalence arguments of~\cite[Section~2]{Pletch1980}. For profinite $\ring \llbracket G \rrbracket$-modules $M$ and $N$, there is a continuous tensor-hom adjunction
\[ (M \widehat{\otimes}_{\ring \llbracket G \rrbracket} N)^{\ast} = \mathrm{Hom}_{\mathrm{cts}}(M \widehat{\otimes}_{\ring \llbracket G \rrbracket} N, \mathbb{Q}/\mathbb{Z}) \xrightarrow{\cong} \mathrm{Hom}_{\ring \llbracket G \rrbracket}(N, M^{\ast}) \]
by~\cite[Lemma~2.4]{bru66}. This adjunction induces canonical isomorphisms on the (co)chain level:
\begin{align*}
\alpha \colon \ring \llbracket G^{n+1} \rrbracket \widehat{\otimes}_{\ring \llbracket G \rrbracket} C &\xrightarrow{\cong} C^n(G, C^{\ast})^{\ast} \, , \\
\beta \colon \big((C \widehat{\otimes}_{\ring} M) \widehat{\otimes}_{\ring \llbracket G \rrbracket} \ring \llbracket G^{n-k+1} \rrbracket\big)^{\ast} &\xrightarrow{\cong} C^{n-k}\big(G, \mathrm{Hom}_{\ring}(M, C^{\ast})\big) \, .
\end{align*}
Note that the domain of $\alpha$ is precisely the continuous $n$-th chain group $C_n(G, C)$. These isomorphisms intertwine the cap and cup product maps via the following diagram:
\begin{center}
\begin{tikzcd}[column sep=large]
    \ring \llbracket G^{n+1} \rrbracket \widehat{\otimes}_{\ring \llbracket G \rrbracket} C \arrow[r, "\Phi_{\frown}"] \arrow[d, "\alpha"', "\cong"] &[-2em] {\mathrm{Hom}_{\ring}\big(\big((C \widehat{\otimes}_{\ring} M) \widehat{\otimes}_{\ring \llbracket G \rrbracket} \ring \llbracket G^{n-k+1} \rrbracket\big)^{\ast}, \, C^k(G, M)^{\ast} \big)} \arrow[d, "{\mathrm{Hom}_{\ring}(\beta^{-1}, \, \id)}"', "\cong"] \\
    C^n(G, C^{\ast})^{\ast} \arrow[r, "\Phi_{\smallsmile}"] & {\mathrm{Hom}_{\ring}\big(C^{n-k}(G, \mathrm{Hom}_{\ring}(M, C^{\ast})), \, C^k(G, M)^{\ast} \big)}
\end{tikzcd}
\end{center}
where $\Phi_{\frown}$ maps a chain $z$ to its associated cap-product pairing $z \frown -$, and $\Phi_{\smallsmile}$ maps a dual cochain $\phi$ to its cup-product pairing $\Psi(\phi)$. A straightforward diagram chase shows that this diagram commutes. 

Because $\alpha$ is a chain complex isomorphism, it induces a canonical isomorphism on homology:
\[ \alpha_{\ast} \colon H_n(G, C) \xrightarrow{\cong} H_{\ring}^n(G, C^{\ast})^{\ast} \, . \]
Let $e_{\mathrm{hom}} \in H_n(G, C)$ be the unique fundamental homology class such that $\alpha_{\ast}(e_{\mathrm{hom}}) = e$. Passing to (co)homology, evaluating the commutative diagram at $e_{\mathrm{hom}}$ demonstrates that the cup-product pairing $\Psi(e)$ is equivalent to the cap-product pairing $e_{\mathrm{hom}} \frown -$ from Corollary~\ref{cor:finiteduality}. Because the latter isomorphism already characterises profinite duality, the result follows.
\end{proof}

\section{Examples and properties of duality groups}\label{sec:exls}

This final section is dedicated to examples and properties of duality groups. By virtue of Corollary~\ref{corC}, it suffices to prove that any given profinite group is a profinite duality group over a profinite ring $\ring$ as it is then a solid duality group over $\cR$. Henceforth, all results in this section are stated entirely within the profinite setting. The proof methods may stem largely from this setting, but the powerful tools from condensed mathematics continue playing a role. All proofs in this section hinge on the Cohen--Macaulay condition because it provides a verifiable condition for duality by Theorem~\ref{thmB}. The first part of the section constructs two novel examples of duality groups while the second part investigates basic group-theoretic properties.

\subsection{Examples}
To prove that our two novel examples are indeed duality groups, we require a technical result.
Recall (\cite[p.~208]{rib10}) that for any profinite group $G$ and any commutative profinite ring $\ring$ the canonical homomorphism $\varepsilon\colon \ring\llbracket G\rrbracket \rightarrow \ring$ sending any generator $g \in G$ to the multiplicative identity $1 \in \ring$ is called the \emph{augmentation map}. Its kernel $I_{\ring} G := \mathrm{Ker}(\varepsilon)$ is called the \emph{augmentation ideal}.

\begin{proposition}\label{prop:cd1groups}
Let $G$ be a profinite group and $\ring$ a commutative profinite ring such that $\cd_{\ring}(G) \leq 1$.
\begin{enumerate}
    \item (\cite[Lemma~6.3.2]{rib10}) If $T \subseteq G$ is a profinite subspace that constitutes a generating set and contains the identity element $1$ of $G$, then
    \[ T-1 = \lbrace t-1 \mid t \in T \rbrace \]
    is a profinite subspace of $I_{\ring}(G)$ that generates $I_{\ring}(G)$ as a profinite $\ring \llbracket G \rrbracket$-module.
    \item There is an inclusion
    \begin{equation}\label{eq:cd1group2}
    H_{\ring}^1(G, \ring \llbracket G \rrbracket) \leq \mathrm{Hom}_{\mathbf{Pro}}\big(T-1, \ring \llbracket G \rrbracket \big)/ \Theta \, ,
    \end{equation}
    where
    \[ \Theta := \lbrace f\colon T-1 \rar \ring \llbracket G \rrbracket \mid \exists x \in \ring \llbracket G \rrbracket \: \forall t \in T: f(t-1) = (t-1)x \rbrace \, . \]
    \item If $\ring \in \lbrace \widehat{\Z}, \Z_p \rbrace$ and $G$ satisfies property (FF), then $H_{\ring}^1(G, \ring \llbracket G \rrbracket)$ is flat as a profinite $\ring$-module.
\end{enumerate}
\end{proposition}

\begin{proof}
We prove the second assertion. Since $\cd_{\ring}(G) \leq 1$, the short exact sequence
\[ 0 \rightarrow I_{\ring}(G) \xrightarrow{\partial} \ring \llbracket G \rrbracket \xrightarrow{\varepsilon} \ring \rightarrow 0 \]
is a projective resolution of $\ring$ by Lemma~\ref{lem:cd}. There is a canonical surjective $\ring \llbracket G \rrbracket$-homomorphism
\[ \pi\colon \big(\ring \llbracket G \rrbracket \big) \llbracket T-1 \rrbracket \rightarrow I_{\ring}(G) \]
such that $\partial \circ \pi\colon \big(\ring \llbracket G \rrbracket \big) \llbracket T-1 \rrbracket \rar \ring \llbracket G \rrbracket$ restricts to the identity on $T-1 \subset \ring \llbracket G \rrbracket\llbracket T-1 \rrbracket $, resulting in an embedding
\[ \mathrm{Hom}_{\ring \llbracket G \rrbracket}(\pi, \ring \llbracket G \rrbracket)\colon \mathrm{Hom}_{\ring \llbracket G \rrbracket}(I_{\ring}(G), \ring \llbracket G \rrbracket) \rar \mathrm{Hom}_{\mathbf{Pro}}(T-1, \ring \llbracket G \rrbracket) \, . \]
Thus, the image of $\mathrm{Hom}_{\ring \llbracket G \rrbracket}(\partial \circ \pi, \ring \llbracket G \rrbracket)$ is given by
\[ \Theta := \{ f\colon T-1 \rar \ring \llbracket G \rrbracket \mid \exists x \in \ring \llbracket G \rrbracket \: \forall t \in T: f(t-1) = (t-1)x \} \, . \]
To prove the third assertion, we show first that $H_{\ring}^1(G, \ring \llbracket G \rrbracket)$ is torsion-free. Assume for now that $G$ cannot be generated by one element. Let $f\colon T-1 \rightarrow \ring \llbracket G \rrbracket$ be a continuous map that does not lie in $\Theta$. Then there exist elements $t_1, t_2 \in T$ for which there is no $x \in \ring \llbracket G \rrbracket$ such that $f(t_i-1) = (t_i-1)x$. For brevity, write $f_i := f(t_i-1)$. Restricting the domain from $T-1$ to $\lbrace t_1-1, t_2-1 \rbrace$, we may identify $f$ with the element $(f_1, f_2)$ in $\ring \llbracket G \rrbracket \oplus \ring \llbracket G \rrbracket$ and $\Theta$ with the closed subgroup
\[ \Theta' := \big\lbrace \big((t-1)x, (t_2-1)x \big) \mid x \in \ring \llbracket G \rrbracket \big\rbrace \leq \ring \llbracket G \rrbracket \oplus \ring \llbracket G \rrbracket \, . \]
Note that a neighbourhood basis of $\ring \llbracket G \rrbracket \oplus \ring \llbracket G \rrbracket$ is given by $\ring_i \llbracket U \rrbracket \oplus \ring_j \llbracket V \rrbracket$ where $U$, $V$ range over all open normal subgroups of $G$ and $\ring_i$, $\ring_j$ range over all open ideals of $\ring$ (cf.~\cite[Section~5.7]{rib10}). It follows from~\cite[Proposition~0.3.3]{wil98} that there are $U$, $V$, $\ring_i$ and $\ring_j$ such that $(f_1, f_2) \notin \Theta' + \ring_i \llbracket U \rrbracket \oplus \ring_j \llbracket V \rrbracket$. Denote by
\[ \pi\colon \ring \llbracket G \rrbracket \oplus \ring \llbracket G \rrbracket \rightarrow \ring \llbracket G/U \rrbracket \oplus \ring \llbracket G/V \rrbracket \]
the canonical projection of profinite $\ring$-modules where the codomain is finitely generated. Because $\ring$ is torsion-free, neither
\[ \pi(f_1, f_2) = \bigg(\sum_{g \in G/U} b_g \cdot g, \sum_{h \in G/V} c_h \cdot h \bigg) \]
nor any of its multiples lie in $\pi(\Theta')$. This demonstrates that no multiple of $f$ lies in $\Theta$ and that $\mathrm{Hom}_{\mathbf{Pro}}(T-1, \ring \llbracket G \rrbracket)/\Theta$ is torsion-free. If $G$ is generated by one element instead of at least two, then $G = \overline{\langle t \rangle}$ is an inverse limit of cyclic groups by~\cite[p.~42]{rib10}. By taking finite quotients as before, one can prove that
\[ \mathrm{Hom}_{\mathbf{Pro}}(T-1, \ring \llbracket G \rrbracket)/\Theta \cong \ring \llbracket G \rrbracket / (t-1)\ring \llbracket G \rrbracket \cong \ring\]
is torsion-free. In either case, $H_{\ring}^1(G, \ring \llbracket G \rrbracket)$ as a submodule of $\mathrm{Hom}_{\mathbf{Pro}}(T-1, \ring \llbracket G \rrbracket)/\Theta$ is torsion-free.

In the case where $\ring = \Z_p$, we infer from Lemma~\ref{lem:profinprojflat} and~\cite[Corollary~2.1.2]{sym00} that $H_{\mathbb{Z}_p}^1(G, \mathbb{Z}_p \llbracket G \rrbracket)$ is flat as a profinite $\mathbb{Z}_p$-module. In the case where $\ring = \widehat{\Z}$, we first use Pontryagin duality. Observe that the discrete $\widehat{\Z}$-modules are exactly the torsion abelian groups and that the profinite $\widehat{\Z}$-modules are exactly the profinite abelian groups. Due to~\cite{osb00}, a discrete $\widehat{\Z}$-module $M$ is injective if and only if it is divisible. By Pontryagin duality, this is equivalent to the Pontryagin dual $M^{\ast}$ being projective, while it is equivalent to $M^{\ast}$ being torsion-free by~\cite[Theorem~2.9.12]{rib10}. Therefore, $H_{\widehat{\Z}}^1(G, \widehat{\Z} \llbracket G \rrbracket)$ is a projective profinite $\widehat{\Z}$-module and thus $\ring$-flat by Lemma~\ref{lem:profinprojflat}.
\end{proof}

\subsubsection{Finitely generated free pronilpotent groups have duality}
Thus far, duality groups have been only constructed over $\mathbb{Z}_p$. We adapt~\cite[Example~4.4.4]{sym00} to construct the very first example of a duality group over $\widehat{\mathbb{Z}}$:

\begin{proposition}\label{prop:freegroups}
Every finitely generated free pronilpotent group is a duality group over $\widehat{\mathbb{Z}}$.
\end{proposition}

\begin{proof}
If $F$ is a free pronilpotent group on a finite set $X = \lbrace x_1, \dots, x_n, 1 \rbrace$, then we prove first that $F$ has finite cohomological dimension over $\widehat{\mathbb{Z}}$. It follows from~\cite[p.~251, Lemma~7.6.3 and Proposition~7.6.7]{rib10} that
\[ \sup_{p \text{ prime}} \cd_{\mathbb{Z}_p}(F) \leq 1 \, . \]
If $A$ is an abelian group and $p$ is a prime number, denote by $A_p$ the subgroup of $A$ in which every element has order $p^n$ for some $n$. Then \cite[Proposition~7.1.1]{rib10} implies that $H_{\widehat{\mathbb{Z}}}^2(F, A)_p = 0$ for any discrete $\widehat{\mathbb{Z}} \llbracket G \rrbracket$-module $A$ and any prime $p$. In particular, $H_{\widehat{\mathbb{Z}}}^2(F, A) = 0$ for any discrete $A$. Because this pertains to any finite $A$ and $H_{\widehat{\mathbb{Z}}}^{\bullet}(F, -)$ commutes with inverse limits by~\cite[Proposition~2.2.4]{rib10}, $H_{\widehat{\mathbb{Z}}}^2(F, M) = 0$ for every profinite $\widehat{\mathbb{Z}} \llbracket F \rrbracket$-module $M$. Lemma~\ref{lem:cd} implies that $\cd_{\widehat{\mathbb{Z}}}(F) \leq 1$.

We are in position to prove that $F$ is a duality group of dimension $1$ over $\widehat{\mathbb{Z}}$. It follows from $\cd_{\widehat{\mathbb{Z}}}(F) \leq 1$, \cite{lim73} and~\cite[Example~4.4.4]{sym00} that there is a projective resolution
\begin{equation}\label{eq:augmentationses}
0 \rightarrow \bigoplus_{i = 1}^n \widehat{\mathbb{Z}} \llbracket F \rrbracket (x_i-1) \xrightarrow{\partial} \widehat{\mathbb{Z}} \llbracket F \rrbracket \xrightarrow{\varepsilon} \widehat{\mathbb{Z}} \rightarrow 0 \, ,
\end{equation}
which arises as follows. We can take $\widehat{\Z} \llbracket F \rrbracket$ to be the (associative) free $\widehat{\Z}$-algebra on the non-commutative variables $\lbrace x_1-1, \dots, x_n -1 \rbrace$. The profinite group $F$ is embedded into $\widehat{\Z} \llbracket F \rrbracket$ via the continuous homomorphism $x_i \mapsto x_i$. Thus, the augmentation map is given by $\varepsilon(x_i) = 1$ and $\varepsilon(1) = 1$. The boundary map $\partial$ is the embedding given by mapping the generator of the $i^{\text{th}}$ component $(0, \dots, 1, \dots, 0)$ to $x_i-1 \in \widehat{\Z} \llbracket F \rrbracket$. As $F$ is thus of type $\mathrm{FP}_{\infty}$ over $\widehat{\Z}$, it satisfies condition (FF) by~\cite[Proposition~4.2.2]{sym00}. According to Proposition~\ref{prop:cd1groups}, $H_{\widehat{\Z}}^1(F, \widehat{\Z} \llbracket F \rrbracket)$ is a flat profinite $\widehat{\Z}$-module. Due to the projective resolution in Equation~\eqref{eq:augmentationses}, this cohomology group is non-vanishing, whence $\cd_{\widehat{\Z}}(F) = 1$. The real benefit of the projective resolution in Equation~\eqref{eq:augmentationses} is that $\mathrm{Hom}_{\widehat{\Z} \llbracket F \rrbracket}(\partial, \widehat{\Z} \llbracket F \rrbracket)$ can be verified to be injective and that thus $H_{\widehat{\Z}}^0(F, \widehat{\Z} \llbracket F \rrbracket) = 0$.
\end{proof}

\subsubsection{A duality group not of type $\mathrm{FP}_{\infty}$}
The following result is useful to when establishing that a given profinite group $G$ is a duality group over the coefficient ring $\Z_p$.

\begin{prop}\label{prop:infiniteporder}
Let $G$ be a profinite group whose order contains a factor of $p^{\infty}$, meaning that $G$ has finite quotients whose orders contain arbitrarily large factors $p^i$. Then $H_{\Z_p}^0(G, \Z_p\llbracket G \rrbracket) = 0$.
\end{prop}

\begin{proof}
One can proceed similarly as in the proof of~\cite[Lemma~3.9]{ghe24}. Assume by contradiction that $H_{\Z_p}^0(G, \Z_p\llbracket G \rrbracket) \neq 0$. By~\cite[p.~171]{rib10}, the completed group ring is given by $\mathbb{Z}_p\llbracket G \rrbracket = \varprojlim_{U \trianglelefteq G \text{ open}} \mathbb{Z}_p[G/U]$. As group cohomology commutes with inverse limits by~\cite[Proposition~2.2.4]{rib10}, there exists an open normal subgroup $U_0 \triangleleft G$ such that $H_{\mathbb{Z}_p}^0(G, \mathbb{Z}_p[G/U_0]) \neq 0$. Because
\[ H_{\Z_p}^0(G, M) = M^G = \{ m \in M \mid \forall g \in G: g \cdot m = m \} \]
for any profinite $\Z_p\llbracket G \rrbracket$-module $M$ and
\[ \Z_p[G/U]^G = \Big\lbrace (b_y)_{y \in G/U} \in \prod_{y \in G/U} \Z_p \mid \exists b \in \Z_p \, \forall y \in G/U: b_y = b \Big\rbrace \]
for any open normal subgroup $U \trianglelefteq G$, there is a nonzero element $b \in \Z_p$ such that
\begin{equation}\label{eq:diagonalelement}
b_{U_0} = (b)_{x \in G/U_0} \in \prod_{x \in G/U_0} \mathbb{Z}_p = \Z_p\llbracket G /U_0\rrbracket^G \, .
\end{equation}
Since $G$ is the inverse limit of its finite quotients, we may construct a countable sequence of open normal subgroups $U_{n+1} \triangleleft U_n \trianglelefteq G$ such that the $p$-factors of the orders $|G/U_n|$ are strictly increasing. Thus, there are elements $b_{U_n} \in \Z_p[G/U_n]^G$ such that for any $m \leq n$ the projection homomorphism $\mathbb{Z}_p[G/U_n]^G \rightarrow \mathbb{Z}_p[G/U_m]^G$ maps $b_{U_n}$ to $b_{U_m}$. It follows from the assumption on the $p$-factors of $|G/U_n|$ that the element $b$ in Equation~\eqref{eq:diagonalelement} is divisible by arbitrarily large powers of $p$. However, this is impossible by~\cite[p.~26--27]{wil98}, and thus  $H_{\Z_p}^0(G, \Z_p\llbracket G \rrbracket) = 0$.
\end{proof}

Below we provide the very first example of a profinite duality group that is not of type $\mathrm{FP}_{\infty}$. This example necessitates the following definition.

\begin{definition}
(\cite[p.~513]{fri08}) If $K$ is a profinite group, then the Frattini subgroup $\Phi(K) \leq K$ is defined as the intersection of all maximal open subgroups of $K$. A homomorphism of profinite groups $\psi\colon J \rar K$ is a Frattini cover if $\psi$ is surjective and $\mathrm{Ker}(\psi) \leq \Phi(J)$. A Frattini cover is universal if $J$ is a projective profinite group, and all profinite groups have a universal Frattini cover.
\end{definition}

\begin{theorem}\label{thm:novelexample}
For any odd prime $p$ define the following profinite group $H$. Let $\mathrm{SL}_p(\mathbb{F}_p)$ act canonically on the vector space $\mathbb{F}_p^p$. For any $n \geq 1$ define the semidirect product $G_n := (\mathbb{F}_p^p)^n \ltimes \mathrm{SL}_p(\mathbb{F}_p)$ with a diagonal action of $\mathrm{SL}_p(\mathbb{F}_p)$. Define the profinite group $G := \prod_{n \geq 1} G_n$ and let $\varphi\colon H \rar G$ denote a universal Frattini cover. Then the following assertions hold.
\begin{enumerate}
\item $H$ is a profinite duality group over $\Z_p$.
\item $H$ is not of type $\mathrm{FP}_{\infty}$ over $\Z_p$.
\end{enumerate}
\end{theorem}

In order to prove that the properties of the example in the above theorem, we require a reformulation of the homological finiteness condition $\mathrm{FP}_1$ over $\Z_p$. By Schanuel's Lemma, $G$ is of type $\mathrm{FP}_1$ over $\ring$ if and only if $I_{\ring} G$ is finitely generated. By Theorem~1 from Damian's paper~\cite{dam11}, $G$ is of type $\mathrm{FP}_1$ over $\Z_p$ if and only if the expression whose terms we define below
\begin{equation}\label{eq:fp1}
\mathrm{sup}_{M \in \mathrm{Irr}(\mathbb{F}_p\llbracket G\rrbracket)} \bigg\lbrace \biggl\lceil \frac{\delta_G(M)+h_G(M)}{r_G(M)} \biggr\rceil \bigg\rbrace
\end{equation}
is finite. Denote by $\mathrm{Irr}(\mathbb{F}_p\llbracket G\rrbracket)$ the set of irreducible finite $\mathbb{F}_p\llbracket G\rrbracket$-modules. If $M \in \mathrm{Irr}(\mathbb{F}_p \llbracket G \rrbracket)$ and $E := \mathrm{End}_{\mathbb{F}_p \llbracket G \rrbracket}(M)$ denotes its field of endomorphisms, then $h_G(M)$ is defined as $\mathrm{dim}_E H_{\mathbb{F}_p}^1(G/C_G(M), M)$ and $r_G(M)$ as
\[ M \cong \mathrm{Hom}_{\mathbb{F}_p \llbracket G \rrbracket}(\mathbb{F}_p\llbracket G\rrbracket, M) \cong E^{r_G(M)} \, . \]
A chief series of $G$ is defined to be a series of closed normal subgroups
\[ {} \dots {} \leq G_2 \leq G_1 \leq G_0 = G \]
such that there is no closed normal subgroup properly contained between any $G_i$ and $G_{i+1}$. A chief factor $H/N$ of $G$ is non-Frattini if $H/N \nleq \Phi(G/N)$. Then $\delta_G(M)$ is defined as the number of non-Frattini chief factors $G$-isomorphic to $M$ in a chief series of $G$. 

In the following proof we change, where appropriate, between the coefficient rings $\mathbb{F}_p$ and $\mathbb{Z}_p$, which is justified by~\cite[Remark~6.2.5]{rib10}.

\begin{proof}[Proof of Theorem~\ref{thm:novelexample}]
To prove the first assertion, we first explain why $H$ has cohomological dimension $1$, then demonstrate that the cohomology of $H$ with finite coefficients is finite and lastly show that $H_{\mathbb{Z}_p}^{\bullet}(H, \mathbb{Z}_p \llbracket H \rrbracket)$ is of the desired form. Thereafter we prove the second assertion.

\textbf{(1)} Because $H$ as a universal Frattini cover is a projective profinite group, $\mathrm{cd}_{\mathbb{Z}_p} (H) \leq 1$ by~\cite[Theorem~7.5.1]{rib10}. By~\cite[Corollaire~2]{ser94}, $\cd_{\mathbb{Z}_p}(H) \neq 0$, so $\cd_{\mathbb{Z}_p}(H) = 1$.

As an intermediate step, we show that $H_{\mathbb{Z}_p}^1(G, M)$ is finite for every finite $\Z_p\llbracket G\rrbracket$-module $M$. It suffices to consider any irreducible finite $M$. If we take $M = \mathbb{F}_p^p$ with the action of the $n^{\text{th}}$ copy of $\mathrm{SL}_p(\mathbb{F}_p)$ in $G$, then $M$ is absolutely irreducible. Taking the appropriate chief series implies that $\delta_G(M)$ is finite. Because $\mathrm{SL}_p(\mathbb{F}_p)$ has trivial centre, $\delta_G(M) = 0$ for any other finite irreducible $\mathbb{F}_p\llbracket G\rrbracket$-module $M$. We restrict our attention to any normal open subgroup $N$ of $G$ that acts trivially on $M$ such that it does not contain any chief factor $G$-isomorphic to $M$. By~\cite[p.~454]{dam11}, $H_{\mathbb{Z}_p}^1(G/N, M) \cong E^{\delta_{G/N}(M)+h_{G_N}(M)}$. As the action on $M$ factors through $G \rar G/N$, $\delta_G(M) = \delta_{G/N}(M)$ and $h_G(M) = h_{G/N}(M)$. By~\cite[p.~455]{dam11}, $E = \mathrm{End}_{\mathbb{F}_p\llbracket G \rrbracket}(M) \cong \mathrm{End}_{\mathbb{F}_p\llbracket G/N \rrbracket}(M)$. Because this holds true for any normal open subgroup $N'$ contained in $N$, we deduce that
\[ H_{\mathbb{Z}_p}^1(G, M) = \varinjlim_{N' \trianglelefteq_o N} H_{\mathbb{Z}_p}^1(G/N', M) \cong E^{\delta_G(M)+h_G(M)} \, . \]
According to~\cite[p.~455]{dam11}, $h_G(M) \leq r_G(M)$ where the latter is finite for a finite module $M$. Hence, $H_{\mathbb{Z}_p}^1(G, M)$ is finite for every finite irreducible $\mathbb{F}_p\llbracket G\rrbracket$-module $M$.

Now we prove that $H_{\mathbb{Z}_p}^k(H, M)$ is finite for every finite $\Z_p\llbracket H\rrbracket$-module $M$ and every $k \geq 0$. Since $\mathrm{cd}_{\Z_p}(H) = 1$, any cohomology group of $H$ of degree greater than $1$ vanishes. Since $H_{\mathbb{Z}_p}^0(H, M)$ is isomorphic to the group of invariants $M^H$, it suffices to prove that $H_{\mathbb{Z}_p}^1(H, M)$ is finite for every finite module $M$. As before, it suffices to consider any finite irreducible $\mathbb{F}_p\llbracket H\rrbracket$-module $M$. If the action on $M$ factors through the surjective homomorphism $H \rar G$, then we infer as above that
\[ H_{\mathbb{Z}_p}^1(H, M) \cong E^{\delta_H(M)+h_H(M)} = E^{\delta_G(M)+h_G(M)} \cong H_{\mathbb{Z}_p}^1(G, M) \, . \]
As the latter cohomology group is finite, the former is so too. Now assume that the action on $M$ does not factor through $H \rar G$. This implies that the Frattini subgroup $\Phi(H)$ acts non-trivially on $M$. In contrast, every non-Frattini chief factor of $H$ is of the form $J/K$ for open normal subgroups $J, K \leq H$ where $\Phi(H) \leq K$ by definition. Thus, $\Phi(H)$ acts trivially on every non-Frattini chief factor of $H$. In particular, $\delta_H(M) = 0$ and $H_{\mathbb{Z}_p}^1(H, M)$ is finite also in this case because $h_H(M) \leq r_H(M)$ as above.

It follows from Proposition~\ref{prop:cd1groups} that $H_{\mathbb{Z}_p}^1(H, \mathbb{Z}_p \llbracket H \rrbracket)$ is a flat profinite $\mathbb{Z}_p$-module. Finally, by Proposition~\ref{prop:infiniteporder}, $H_{\mathbb{Z}_p}^0(H, \mathbb{Z}_p \llbracket H \rrbracket) = 0$ because $p^{\infty}$ divides the order of $G$ and thus the order of $H$.

\textbf{(2)} It suffices to show that $H$ is not of type $\mathrm{FP}_1$ over $\Z_p$ to conclude that it is not of type $\mathrm{FP}_{\infty}$. Since there is a surjective homomorphism $\varphi\colon H \rar G$, there is a surjective homomorphism of $\Z_p\llbracket G\rrbracket$-modules $\Z_p\llbracket H\rrbracket \rar \Z_p\llbracket G\rrbracket$. A diagram chase reveals that the latter homomorphism restricts to a surjective homomorphism $I_{\Z_p}H \rar I_{\Z_p}G$. If $I_{\Z_p}G$ is not finitely generated, then the same holds true for $I_{\Z_p}H$. Thus, it suffices to show that $G$ is not of type $\mathrm{FP}_1$ over $\Z_p$. If we take the absolutely irreducible $\mathbb{F}_p\llbracket G\rrbracket$-module $M = \mathbb{F}_p^p$ with the action of the $n^{\text{th}}$ copy of $\mathrm{SL}_p(\mathbb{F}_p)$ in $G$ as above and take an appropriate chief series, then $\delta_G(M) = n$ and $r_G(M) = p$. Therefore, the supremum in Equation~\eqref{eq:fp1} does not exist and $G$ is not of type $\mathrm{FP}_1$ over $\Z_p$.
\end{proof}

\subsection{Closure properties}

We end this paper by establishing group-theoretic closure properties of duality groups. 

\subsubsection{Subgroup closure}
We start by proving that duality groups are closed under taking open subgroups, thus generalising \cite[Proposition~4.4.1]{sym00} to any profinite coefficient ring.
\begin{proposition}\label{prop:subgroupclosed}
Let $G$ be a profinite group, $H$ an open subgroup and $\ring$ a profinite commutative ring such that $\cd_{\ring}(G) < \infty$ and $G$ satisfy condition (FF). Then $G$ is a duality group over $\ring$ if and only if $H$ is a duality group over $\ring$.
\end{proposition}

\begin{proof}
First we prove that $H_{\ring}^k(H, M)$ is finite for any finite $M$ and $k$ if and only if $H_{\ring}^k(G, M)$ is so too. It follows from~\cite[pp.~370--371]{sym00} that coinduction $\mathrm{Hom}_{\ring \llbracket H \rrbracket}(\ring \llbracket G \rrbracket, -)$ is a functor mapping profinite $\ring \llbracket H \rrbracket$-modules to profinite $\ring \llbracket G \rrbracket$-modules that is left adjoint to the restriction functor. One can consider the long exact sequence in $H_{\ring}^{\bullet}(G, -)$ associated to the unit and any finite $\ring \llbracket G \rrbracket$-module $M$
\[ 0 \rightarrow M \rightarrow \mathrm{Hom}_{\ring \llbracket H \rrbracket}(\ring \llbracket G \rrbracket, M) \rightarrow \mathrm{Coker} \rightarrow 0 \, . \]
The equivalence of finite cohomology groups can be deduced as in the proof of~\cite[Proposition~1.8]{Pletch1980} using Shapiro's Lemma from~\cite[Lemma~4.2.8]{sym00}.

It can be deduced as in~\cite[Proposition~VIII.2.4]{bro82} that $\cd_{\ring}(H) = \cd_{\ring}(G)$. Because $\ring \llbracket G \rrbracket$ is a projective profinite $\ring \llbracket H \rrbracket$-module by~\cite[Proposition~5.7.1]{rib10}, the coinduction functor $\mathrm{Hom}_{\ring \llbracket H \rrbracket}(\ring \llbracket G \rrbracket, -)$ is exact and its left-adjoint restriction functor preserves projectives. In particular, any projective resolution $P_{\bullet} \rightarrow \ring$ of profinite $\ring \llbracket G \rrbracket$-modules is also one of $\ring \llbracket H \rrbracket$-modules. This proves that $\cd_{\ring}(H) \leq \cd_{\ring}(G)$. Let us prove the converse inequality. As $\cd_{\ring}(G) = n < \infty$, there is by Lemma~\ref{lem:cd} a profinite $\ring \llbracket G \rrbracket$-module $M$ such that $H_{\ring}^n(G, M) \neq 0$. The long exact sequence resulting from the surjection $\big(\ring \llbracket G \rrbracket \big) \llbracket M \rrbracket \rightarrow M$ shows that $H_{\ring}^n \big(G, (\ring \llbracket G \rrbracket)\llbracket M \rrbracket \big) \neq 0$. It follows from \cite[Proposition~2.2.4 and Proposition~5.2.2]{rib10} that $H_{\ring}^n \big(G, \ring \llbracket G \rrbracket \big) \neq 0$. We conclude from~\cite[Lemma~4.2.8]{sym00} that
\[ H_{\ring}^n \big(G, \ring \llbracket G \rrbracket \big) \cong H_{\ring}^n \big(H, \ring \llbracket H \rrbracket \big) \]
as profinite $\ring \llbracket H \rrbracket$-modules. This isomorphism does not only prove the converse inequality. It also proves that the $H_{\ring}^n \big(G, \ring \llbracket G \rrbracket \big)$ is $\ring$-flat if and only if $H_{\ring}^n \big(H, \ring \llbracket H \rrbracket \big)$ is so. It also proves that $H_{\ring}^k \big(G, \ring \llbracket G \rrbracket \big)$ vanishes for $k \neq n = \cd_{\ring}(G)$ if and only if $H_{\ring}^k \big(H, \ring \llbracket H \rrbracket \big)$ does so.
\end{proof}

\subsubsection{Extension closure}
We prove that duality groups are closed under taking extensions, thus generalising~\cite[Theorem~3.9]{Pletch1980} to arbitary profinite rings. The proof of this is analogous to classical proofs of extension closure, which rely on a Lyndon--Hochschild--Serre spectral sequence. This tool is of independent interest:

\begin{lemma}[Solid Lyndon--Hochschild--Serre spectral sequence]
    Let $0 \rightarrow N \rightarrow G \rightarrow Q \rightarrow 0$ be a short exact sequence of profinite groups. Then there is a spectral sequence
    \begin{equation}\label{eq:lyndonsequence1}
E_2^{pq}=\mathbf{H}_{\cR}^p \big(Q, \mathbf{H}_{\cR}^q(N, M)\big) \Rightarrow \mathbf{H}_{\cR}^{p+q}(G, M),
\end{equation}
which converges for any solid $G$-module $M$. 
\end{lemma}
\begin{proof}
    Let $Q_{\bullet} \rightarrow {\cR}$ be a projective resolution of solid $\ring(Q)$-modules and $P_{\bullet} \rightarrow {\cR}$ a projective resolution of solid $\rg$-modules. It follows from~\cite[Proposition~5.7.1]{rib10} and Lemma~\ref{lem:profinitembeddings} that $\rg$ is a projective solid $\ring(N)$-module. Coinduction $\mathrm{Hom}_{\ring(N)}(\rg, -)$ is exact and restriction as its left-adjoint preserves projectives. Thus, $P_{\bullet} \rightarrow {\cR}$ is a projective resolution of solid $\ring(N)$-modules from which we obtain as in~\cite[p.~384]{sym00} a first quadrant double complex
\begin{equation}\label{eq:lyndoncomplex}
\underline{\mathrm{Hom}}_{\ring(Q)}\big(Q_{\bullet}, \underline{\mathrm{Hom}}_{\ring(N)}(P_{\bullet}, M) \big)
\end{equation}
for any solid $\rg$-module $M$. This gives rise to the Lyndon--Hochschild--Serre spectral sequence in the standard way.
\end{proof}

By virtue of Theorem~\ref{thm:profintosolid} and Lemma~\ref{lem:profincohomgroups}, this lemma carries over to the profinite setting:

\begin{corollary}[Profinite Lyndon--Hochschild--Serre spectral sequence]\label{cor:LHSprof}
     Let $0 \rightarrow N \rightarrow G \rightarrow Q \rightarrow 0$ be a short exact sequence of profinite groups satisfying condition (FF). Then there is a spectral sequence
     \begin{equation}\label{eq:lyndonsequence2}
H_{\ring}^k \big(Q, H_{\ring}^m(N, M)\big) \Rightarrow H_{\ring}^{k+m}(G, M)
\end{equation}
converging for any discrete or profinite $G$-module $M$.
\end{corollary}

To obtain this spectral sequence for profinite $M$ entirely via methods from the profinite setting as in~\cite{Pletch1980}, $\mathrm{Hom}_{\ring \llbracket N \rrbracket}(P_{\bullet}, M)$ from Equation~\eqref{eq:lyndoncomplex} needs to be a profinite module. This is achieved in~\cite[Theorem~4.2.6]{sym00} by imposing that $N$ is open and thus of finite index in $G$. This does not suffice for the next proposition which proves closure of profinite duality groups under extensions. If one assumes the quotient $Q$ to be a duality group, it has $\cd_{\ring}(Q) < \infty$ and is hence torsion-free by~\cite[Proof of Corollary~VIII.2.5]{bro82} and~\cite[Remark~6.2.5]{rib10}. This implies that $N$ is not of finite index. Condensed mathematics allows us to bypass these issues and show:

\begin{proposition}\label{lem:extensionclosed}
Let $0 \rightarrow N \rightarrow G \rightarrow Q \rightarrow 0$ be a short exact sequence of profinite groups such that $N$ and $Q$ satisfy condition (FF). If $\ring$ is a profinite commutative ring over which $N$ is a duality group of dimension $n$ with dualising module $D_{\ring, N}$ and $Q$ is a duality group of dimension $q$ with dualising module $D_{\ring, Q}$, then $G$ is a duality group of dimension $n+q$ over $\ring$ with dualising module $D_{\ring, N}\widehat{\otimes}_\ring D_{\ring, Q}$. 
\end{proposition}

\begin{proof}
The proof follows and rectifies that of~\cite[Theorem~3.9]{Pletch1980}. We write $D_{\ring, N}^{\bullet}$ and $D_{\ring, Q}^{\bullet}$ for the dualising complexes of $N$ and $Q$, respectively. By the profinite Lyndon--Hochschild--Serre spectral sequence from Corollary~\ref{cor:LHSprof}, we see that $N$ and $Q$ satisfying condition (FF) implies that $G$ also satisfies this condition. Writing $\cd_{\ring}(N) = n$ and $\cd_{\ring}(Q) = q$, we conclude by Lemma~\ref{lem:cd} that $\cd_{\ring}(G) \leq n+q$.

We show now that $H_{\ring}^{\bullet}(G, \ring \llbracket G \rrbracket)$  is concentrated in degree $n+q$ and that
\[ H_{\ring}^{n+q}(G, \ring \llbracket G \rrbracket)\cong D_{\ring, N}\widehat{\otimes}_\ring D_{\ring, Q} \, , \]
from which the proposition will follow. Using that $\ring \llbracket G \rrbracket$ is an induced profinite $\ring \llbracket N \rrbracket$-module \cite[Proposition~5.7.1]{rib10} and that $N$ is a duality group, we conclude by~\cite[Remark~3.3.3]{sym00} that
\[ H_{\ring}^k(N, \ring \llbracket G \rrbracket) \cong H_{n-k}^{\ring}(N, D_{\ring, N} \widehat{\otimes}_{\ring} \ring \llbracket G \rrbracket) = 0 \quad \text{for } k \neq n \, . \]
Therefore, the spectral sequence from Equation~\eqref{eq:lyndonsequence2} collapses to a single line
\begin{equation}\label{eq:lyndonsequence3}
H_{\ring}^k \big(Q, H_{\ring}^n(N, \ring \llbracket G \rrbracket)\big) \cong H_{\ring}^{k+n}(G, \ring \llbracket G \rrbracket) \, .
\end{equation}
To proceed, we need to show that the duality isomorphism $H_{\ring}^n(N, \ring \llbracket G \rrbracket) \cong D_{\ring, N} \widehat{\otimes}_{\ring\llbracket N\rrbracket} \ring \llbracket G \rrbracket$ is one of profinite $\ring \llbracket Q \rrbracket$-modules, which does not follow from profinite arguments and Theorem~\ref{thmB} directly.

Instead, we establish this by examining the solid origins of the duality isomorphism (Lemma~\ref{lem:corelemma}). Because $N$ is a normal subgroup of $G$, the conjugation action of $G$ on $N$ endows the solid group ring $\ring(N)$ with the canonical structure of a $\rg$-bimodule. This makes $\ring(N)$ into an algebra object in $\SolidLG$. Consequently, for any solid $\rg$-modules $L$ and $M$, both $L \otimes^\solid_{\ring(N)} M$ and $\mathrm{R\underline{Hom}}_{\ring(N)}(L, M)$ inherit the structure of solid $\rg$-modules from the tensor product and internal hom in $\SolidLG$,\footnote{Classically, the action of $G$ on $L\otimes_N M$ is given by $g (l\otimes m)= (gl)\otimes(gm)$. By virtue of normality of $N$, one checks that this action leaves the submodule of $L\otimes M$ spanned by elements of the form $n^{-1}l \otimes m - l\otimes nm$ for $n\in N$ invariant.} and the unit and counit for the hom-tensor adjunction are $\rg$-linear maps. This in particular implies that the dualising complex $D_{\cR, N}^\bullet = \mathrm{R\underline{Hom}}_{\ring(N)}(\cR, \ring(N))$ is a $\rg$-module, and that the composition map
\[ D_{\cR, N}^\bullet \otimes_{\ring(N)}^{\solid} \mathrm{R\underline{Hom}}_{\ring(N)}(\ring(N), M)\to \mathrm{R\underline{Hom}}_{\ring(N)}({\cR}, M) \]
is $\rg$-linear. This then implies (see Equation~\eqref{eq:defofphi}) that $N$'s duality isomorphism is $\rg$-linear when applied to $\rg$-modules. Using this when $M = \rg$ and taking the $n^{\text{th}}$ solid homology yields, by Theorem~\ref{thm:profintosolid}, an isomorphism of profinite $\ring \llbracket G \rrbracket$-modules:
\[ D_{\ring, N} \widehat{\otimes}_{\ring\llbracket N \rrbracket} \ring \llbracket G \rrbracket \cong H_{\ring}^n(N, \ring \llbracket G \rrbracket) . \]
This in fact descends to an isomorphism of profinite $\cgr{Q}$-modules as the (left) action of $\cgr{N}$ on both sides is trivial. 

Using this $\ring \llbracket Q \rrbracket$-isomorphism and Equation~\eqref{eq:lyndonsequence3}, we conclude that
\[ H_{\ring}^{n+q}(G, \ring \llbracket G \rrbracket) \cong H_{\ring}^q \big(Q, H_{\ring}^n(N, \ring \llbracket G \rrbracket) \big) \cong D_{\ring, Q} \widehat{\otimes}_{\ring} D_{\ring, N} \, . \]
This proves that $\cd_{\ring}(G) = n+q$ and that $D_{\ring, G} = H_{\ring}^{n+q}(G, \ring \llbracket G \rrbracket)$ is a flat profinite $\ring$-module.
\end{proof}

\subsubsection{Amalgamation closure}
In this final subsection, we prove that duality groups are closed under taking profinite amalgamated products, thus generalising~\cite[Theorem~3.5]{Pletch1980a} to arbitrary profinite rings.

\begin{proposition}\label{lem:amalgamclosed}
Let $G = G_1 \bigsqcup_L G_2$ be a profinite amalgamated product and let $\ring$ be a commutative profinite ring over which $G_1$, $G_2$ are duality groups of dimension $n$ and $L$ is a duality group of dimension $n-1$, all satisfying condition (FF). Then $G$ is a duality group over $\ring$.
\end{proposition}

\begin{proof}
It follows from~\cite[Theorem~1.13]{gil74} that there is a Mayer--Vietoris sequence
\[ {} \dots {} \rightarrow H_{\ring}^{k-1}(L, M) \rightarrow H_{\ring}^k(G, M) \rightarrow H_{\ring}^k(G_1, M) \oplus H_{\ring}^k(G_2, M) \rightarrow H_{\ring}^k(L, M) \rightarrow {} \dots {} \]
for any discrete $\ring \llbracket G \rrbracket$-module $M$. Thus, $H_{\ring}^k(G, M)$ is finite for any finite $M$ and any $k$. By~\cite[Proposition~2.2.4]{rib10}, the Mayer--Vietoris sequence also pertains to any profinite $\ring \llbracket G \rrbracket$-module $M$. Therefore, $H_{\ring}^k(G, \ring \llbracket G \rrbracket) = 0$ for $n \neq k$ and
\begin{equation}\label{eq:dualmodses}
0 \rightarrow H_{\ring}^{n-1}(L, \ring \llbracket G \rrbracket) \rightarrow H_{\ring}^n(G, \ring \llbracket G \rrbracket) \rightarrow H_{\ring}^n(G_1, \ring \llbracket G \rrbracket) \oplus H_{\ring}^n(G_2, \ring \llbracket G \rrbracket) \rightarrow 0
\end{equation}
is a short exact sequence. Denote by $H$ any of the subgroups $L$, $G_1$ or $G_2$. By~\cite[Proposition~5.2.2 and Proposition~5.7.1]{rib10}, $H_{\ring}^n(H, \ring \llbracket G \rrbracket)$ can be written as an inverse limit of finite products of $D_{\ring, H}$. Therefore, $H_{\ring}^n(G, \ring \llbracket G \rrbracket)$ is nonzero and $\cd_{\ring}(G) = n$. We know that $D_{\ring, H}$ is a flat profinite $\ring$-module. Because completed tensor products commute with inverse limits and finite products by~\cite[Lemma~5.5.2 and Proposition~5.5.3]{rib10}, $H_{\ring}^n(H, \ring \llbracket G \rrbracket)$ is a flat and thus projective profinite $\ring$-module by Lemma~\ref{lem:profinprojflat}. We deduce from Equation~\eqref{eq:dualmodses} and Lemma~\ref{lem:profinprojflat} that $D_{\ring, G} = H_{\ring}^n(G, \ring \llbracket G \rrbracket)$ is a flat profinite $\ring$-module.
\end{proof}

\bibliographystyle{alpha}
\bibliography{library.bib}
	
\end{document}